\documentclass[11pt, a4paper]{amsart}

\usepackage{amsmath,amssymb,amsfonts,mathrsfs,amsthm,bbm,enumitem, graphicx, csquotes, color, mathtools}

\DeclareMathOperator{\var}{var}

\numberwithin{equation}{section}

\theoremstyle{plain}
\newtheorem{theorem}{Theorem}[section]
\newtheorem{lemma}[theorem]{Lemma}
\newtheorem{prop}[theorem]{Proposition}

\newtheorem{remark}[theorem]{Remark}

\newtheorem{maintheorem}{Theorem}

\theoremstyle{definition}

\newtheorem{defn}[theorem]{Definition}

\newcommand{\N}{\mathbbm{N}}
\newcommand{\Z}{\mathbbm{Z}}

\newcommand{\R}{\mathbbm{R}}
\newcommand{\een}{\mathbbm{1}}

\newcommand{\tyk}{\mathcal}

\def\Eah{\tyk{E}_{\textup{ah}}}
\def\E{\tyk{H}_{\textup{ea}}}

\def\Hio{\tyk{H}_{\textup{io}}}
\def\Rio{\tyk{R}_{\textup{io}}}
\def\newRio{\widehat{\mathcal{R}}_\mathrm{io}}

\newcommand{\dd}{\mathrm{d}}
\newcommand{\para}[1]{\left( #1 \right )}
\newcommand{\tub}[1]{\left\{#1\right\}}
\newcommand{\num}[1]{\left | #1\right |}
\newcommand{\norm}[1]{\left\Vert #1\right\Vert}
\newcommand{\ip}[2]{\left < #1,#2 \right>}

\newcommand{\ear}{\tyk{R}_{\textup{ea}}}

\newcommand{\abs}[1]{\left\lvert#1\right\rvert}

\DeclareMathOperator{\udimb}{\overline{dim_B}}

\renewcommand{\epsilon}{\varepsilon}

\allowdisplaybreaks[4]

\begin{document}
\title[On shrinking targets \& self-returning points]{On shrinking targets and self-returning points}

\author{Maxim Kirsebom}
\address{Maxim Kirsebom, University of Hamburg, Department of Mathematics, Bundesstrasse 55, 20146 Hamburg, Germany}
\email{maxim.kirsebom@uni-hamburg.de}

\author{Philipp Kunde}
\address{Philipp Kunde, Pennsylvania State University, Department of Mathematics, McAllister Building, State College, PA 16802, USA}
\email{pkunde.math@gmail.com}

\author{Tomas Persson}
\address{Tomas Persson, Centre for Mathematical Sciences, Lund University, Box 118, 221 00 Lund, Sweden}
\email{tomasp@maths.lth.se}

\thanks{We thank Victor Ufnarovski for proving
  Lemma~\ref{lem:ufnarovski} for us. We also thank both Simon
  Baker for bringing the references \cite{Baker}, \cite{Chang} to
  our attention and Dmitry Kleinbock for asking about an
  eventually-always version of Boshernitzan's Theorem
  \ref{thm:Bos} which led to our Theorem \ref{theo:EARgen}. We
  acknowledge financial support by the Hamburg--Lund Funding
  Program 2018 which made several mutual research visits
  possible. P.~K.\ acknowledges financial support from a DFG
  Forschungsstipendium under Grant No.\ 405305501.}


\begin{abstract}
  We consider the set $\Rio$ of points returning infinitely many
  times to a sequence of shrinking targets around
  themselves. Under additional assumptions we improve
  Boshernitzan's pioneering result on the speed of recurrence. In
  the case of the doubling map as well as some linear maps on the
  $d$-dimensional torus, we even obtain a dichotomy condition for
  $\Rio$ to have measure zero or one. Moreover, we study the set
  of points eventually always returning and prove an analogue of
  Boshernitzan's result in similar generality.
\end{abstract}

\maketitle

\section{Introduction}

Let $(X,\mathcal{B},\mu,T)$ be a measure preserving system
equipped with a compatible metric $d$, i.\,e.\ a metric such that
open subsets of $X$ are measurable. We consider a sequence
$\tub{B(y,r_n)}_{n=1}^{\infty}$ of balls in $X$ with center $y$
and radius $r_n$. We will refer to the balls as shrinking targets
since the interesting questions arise when $r_n \to 0$ although
this is not a formal requirement.
Classical shrinking target questions focus on the set of $x\in
X$, whose $n$'th iterate under $T$ hits $B(y,r_n)$ for infinitely
many $n$. That is, the set
\begin{equation*}
  \Hio = \Hio(y,r_n) := \tub{\, x \in X : T^{n}(x)\in B(y,r_n)
    \text{ for $\infty$ many } n \in \N \,}.
\end{equation*}
For many dynamical systems, the measure as well as dimension of
this set is well understood under certain assumptions on the
measure of the shrinking targets (see \cite{Athreya},
\cite{ChernovKleinbock}, and references therein for examples). A
different and interesting question arises when we do not consider
one fixed center for the shrinking targets, but instead consider
the points that return infinitely many times to a sequence of
shrinking targets around \emph{themselves}.  That is, the set
\begin{equation*}
  \Rio = \Rio(r_n) := \tub{\, x \in X: T^n(x)\in B(x,r_n) \text{ for
      $\infty$ many } n \in \N \,}.
\end{equation*}
If the invariant measure $\mu$ is nonuniform, then the measure of the targets depends on their location and we also consider the set
\[
\newRio = \newRio (M_n) := \{\, x : T^n (x) \in B(x, r_n(x)) \text{ for }
\infty \text{ many } n \in \mathbbm{N} \,\},
\]
where $r_n (x)$ is such that $\mu (B(x,r_n(x))) = M_n$.

Another interesting set to consider is the
\emph{eventually-always} analogue of $\Rio$ which is defined as
\begin{align*}
  \ear & =\ear(r_m) \\ & :=\bigl\{ \, x\in X: \exists\, m_0\in
  \N,\ \forall\, m\geq m_0: \bigl\{ T^{k}(x)
  \bigr\}_{k=1}^{m}\cap B(x,r_m)\neq \emptyset \,\bigr\},
\end{align*}
i.\,e., the set of points whose sufficiently long orbit always
hits a sequence of shrinking targets around themselves. In
addition to obtaining a result on $\ear$ in broad generality, we
study the measure of $\Rio$, $\newRio$, and $\ear$ for certain classes of
dynamical systems on the unit interval as well as some linear
maps on the $d$-dimensional torus. Note also that the
eventually-always analogue of $\Hio$, denoted by $\E$ or
sometimes by $\Eah$, was investigated for similar dynamical
systems by the authors in \cite{KKP1} and by Kleinbock,
Konstantoulas and Richter in \cite{KKR}.

\subsection{Known results about $\Rio$, $\widehat{\mathcal{R}}_{\mathbf{io}}$, and $\ear$}

Generally speaking we are interested in the sizes of these sets
and how their sizes depend on the measure of the targets. By
\enquote{size} we mean measure, but if the measure of the set is
zero it is interesting to determine the dimension of the set to
get a more nuanced picture of how small it is. In the setting of
$\beta$-transformations, the Hausdorff dimension of $\ear$ was
computed by Zheng and Wu in \cite{ZhengWu}.

In this paper we focus on the measure of $\Rio$, $\newRio$, and $\ear$. The
study of self-returning points invariably starts with the
Poincar\'e Recurrence Theorem, which may be stated as follows.

\begin{theorem}[{Carath\'{e}odory, 1919 \cite{Caratheodory}}]
  Let $(X,d)$ be a separable metric space and let $\mu$ be a
  finite $T$-invariant Borel measure.  For $\mu$-almost every
  $x\in X$, there exists a subsequence $n_k$ such that
  $T^{n_k}(x)\to x$ as $k \to \infty$.
\end{theorem}

The conclusion of the theorem can also be rewritten as
\begin{multline*}
  \mu (\{\, x\in X:\exists\, (r_n(x))_{n\in \mathbbm{N}} \text{
    s.t. } r_n (x) \to 0 \text{ as } n \to \infty \text{ and }
  \\ T^{n}(x)\in B (x,r_{n}(x)) \text{ for $\infty$ many } n\in\N
  \,\})=1.
\end{multline*}
We see that the sequence $r_n$ is allowed to depend on the point
$x$ and the rate of $r_n\to 0$ may also be arbitrarily slow. It
is natural to ask under which circumstances there exists a
certain rate on $r_n\to 0$ which is uniform across all $x\in X$
and which maintains full measure as above.  In his pioneering
paper \cite{Bos}, Boshernitzan gave the following answer to this
question.
\begin{theorem}[{\cite[Theorem~1.2]{Bos}}]
  \label{thm:Bos}
  Let $(X,T,\mu)$ be a measure preserving system equipped with a
  metric $d$. Let $H_{\alpha}$ denote the Hausdorff
  $\alpha$-measure for some $\alpha>0$ and assume that
  $H_{\alpha}$ is $\sigma$-finite on $X$. Then for $\mu$-almost
  every $x\in X$ we have
  \begin{equation}\label{bos1}
    \liminf_{n\geq 1} \Bigl\{n^{\frac{1}{\alpha}}d(T^n(x),x)
    \Bigr\} < \infty.
  \end{equation}
  Furthermore, if $H_{\alpha}(X)=0$, then for $\mu$-almost every
  $x\in X$ we have
  \begin{equation}\label{bos2}
    \liminf_{n\geq 1} \Bigl\{n^{\frac{1}{\alpha}}d (T^n (x),x)
    \Bigr\}=0.
  \end{equation}
\end{theorem}
Note that in general, if $\alpha>\dim_\textup{H}(X)$ then
$H_{\alpha}(X)=0$ and $H_{\alpha}$ is (trivially) $\sigma$-finite
on $X$. In many cases, for example when $X=\R^{k}$ with the
Euclidian metric, we have that $H_{\alpha}$ is $\sigma$-finite on
$X$ if and only if $\alpha\geq \dim_\textup{H}(X)$. In this paper
we will mainly focus on interval maps, and hence we take a closer
look at Theorem~\ref{thm:Bos} when $X=[0,1]$. Then $H_{\alpha}$
is $\sigma$-finite on $X$ for all $\alpha\geq 1$ and
$H_{\alpha}(X)=0$ for all $\alpha>1$. Statement \eqref{bos1} then
corresponds to the case $\alpha=1$ and can be reformulated as:
For $\mu$-almost every $x\in X$ there exists a constant
$\kappa(x)>0$ such that if $r_n(x)\geq \frac{\kappa(x)}{n}$, then
$T^n(x)\in B (x,r_n(x))$ for infinitely many $n\in\N$, i.\ e.\
\begin{equation*}
 \mu(\{\, x: T^n(x)\in B (x,r_n(x)) \text{ for } \infty
     \text{ many } n\in\N \,\}) = 1.
\end{equation*}
Statement \eqref{bos2} corresponds to the case $\alpha>1$ and
enables us to get rid of the $x$-dependence of the radii by
decreasing the shrinking rate slightly. As a consequence,
for any $\beta<1$ and any $\kappa>0$, we have that
\begin{equation*} 
  r_n\geq \frac{\kappa}{n^{\beta}}\enskip\Rightarrow\enskip
  \mu(\Rio(r_n))=1.
\end{equation*}
Boshernitzan's result is surprisingly strong given its level of
generality. Since then much work has been done on the topic of
self-returning points. However, it appears that even with much
stronger assumptions on the system, few improvements of
Boshernitzans rate of $n^{\frac{1}{\alpha}}$ have been
obtained. As far as we know, the only improvements were obtained
by Pawelec \cite[Theorem~3.1]{Paw2017}; Chang, Wu and Wu
\cite{Chang}; Baker and Farmer \cite{Baker}; and recently by
Hussein, Li, Simmons and Wang \cite{Husseinetal}.

Pawelec proved that Boshernitzan's rate can be improved by a
factor $(\log\log n)^{\frac{1}{\alpha}}$ under the assumption of
exponential mixing as well as a regularity assumption on the
invariant measure which is related to the value of
$\alpha$.

Chang, Wu and Wu as well as Baker and Farmer obtained
improvements to Boshernitzan's result for self-similar
sets. Hussein, Li, Simmons and Wang obtained a dichotomy result
for some expanding conformal systems, including piecewise
expanding maps with an absolutely continuous invariant
measure. For such piecewise expanding maps, their result is that
$\mu(\Rio(r_n)) = 1$ if and only if $\sum r_n = \infty$, and
otherwise $\mu (\Rio (r_n)) = 0$.

A different perspective on the Boshernitzan result is given
through the strong connection between the speed with which a
typical point returns close to itself, and the local property of
the measure. Let
\[
\tau_r (x) = \inf \{\, n \in \mathbbm{N} : d (T^n (x), x) < r \,\}
\]
and
\[
\underline{R} (x) = \liminf_{r \to 0} \frac{\log \tau_r (x)}{-
  \log r} \qquad \text{and} \qquad \overline{R} (x) = \limsup_{r
  \to 0} \frac{\log \tau_r (x)}{- \log r}.
\]
Barreira and Saussol \cite{BarreiraSaussol} proved that if $\mu$
is an invariant probability measure, then for $\mu$ almost every
$x$ holds
\begin{equation} \label{eq:barreirasaussol}
  \underline{R} (x) \leq \underline{d}_\mu (x) \qquad \text{and}
  \qquad \overline{R} (x) \leq \overline{d}_\mu (x),
\end{equation}
where $\underline{d}_\mu (x)$ is the lower pointwise dimension of
$\mu$ at $x$ and $\overline{d}_\mu (x)$ is the upper pointwise
dimension of $\mu$ at $x$, defined by
\[
\underline{d}_\mu (x) = \liminf_{r \to 0} \frac{\log \mu
  (B(x,r))}{\log r} \qquad \text{and} \qquad \overline{d}_\mu (x)
= \limsup_{r \to 0} \frac{\log \mu (B(x,r))}{\log r}.
\]

Suppose for simplicity that $x$ is a point such that
$\underline{d}_\mu (x) = \overline{d}_\mu (x) = s$ and
\eqref{eq:barreirasaussol} holds. Then for any $\varepsilon > 0$
we have $\tau_r (x) \leq r^{-(s+\varepsilon)}$ for small
$r$. This tells us that if $r_n = n^{-\alpha}$, then $d(T^n(x),x)
< r_n$ holds for infinitely many $n$ if $\alpha
(s+\varepsilon) \leq 1$. In other words, for any $\varepsilon > 0$
we have that $d(T^n(x),x) < n^{-\frac{1}{s+\varepsilon}}$ holds
for infinitely many $n$. If $\underline{d}_\mu (x) =
\overline{d}_{\mu} (x) = s$ for $\mu$-almost every $x$, the
conclusion obviously holds almost surely. Formulated in another
way, if $\underline{d}_\mu (x) = \overline{d}_{\mu} (x) = s$ for
$\mu$-almost every $x$, then $\mu(\Rio (n^{-\alpha})) = 1$ if $\alpha
< \frac{1}{s}$.


Hence, the result of Barreira and Saussol is similar to the
result of Boshernitzan. However, the result of Barreira and
Saussol gives information about the return time $\tau_r$, which
the result of Boshernitzan does not.

\subsection{Outline of the paper}

In this paper we prove various strengthenings of the known
results on the measure of $\Rio$. In Theorem~\ref{thm:riorate},
we show that the rate given by Pawelec can be significantly
improved for a large class of interval maps, including some
quadratic maps. For this result, we need only an assumption on
decay of correlations and that the invariant measure is
absolutely continuous with respect to Lebesgue measure. A similar
result is in Theorem~\ref{thm:riorate2}, where we are also to
obtain sufficient conditions for $\mu(\newRio)=1$ for systems
with an invariant measure that is not absolutely continuous with
respect to Lebesgue measure, but satisfies a regularity
assumption of the same type as used by Pawelec.
 
In Theorem~\ref{thm:RioMeasZero} we give general sufficient
conditions for $\Rio$ and $\newRio$ to be of zero measure under mixing
assumptions.

We then turn our attention to the case of the doubling map as
well as some linear maps on the $d$-dimensional torus for which
we are able to prove in Theorem \ref{thm:doublingdichotomy} an exact dichotomy for when $\Rio$ is of
zero and full measure. In addition to rotations\footnote{Note
that for a rotation $R_\alpha$, we have $|R_\alpha^n (x) - x| =
|R_\alpha^n (0) + x - x| = |R_\alpha^n (0)|$. Hence $|R_\alpha^n
(x) - x| < r_n$ for infinitely many $n$ iff $|R_\alpha^n (0)| <
r_n$ for infinitely many $n$. Hence there is a kind of dichotomy
which gives either $\Rio = \emptyset$ or $\Rio = \mathbbm{S}^1=X$,
depending on a condition on $\alpha$ and $r_n$. By the
Duffin--Schaeffer conjecture (now a theorem of Koukoulopoulus and
Maynard \cite{KM}), for almost all $\alpha$, we have $|R_\alpha
(0)| < r_n$ for infinitely many $n$ iff $\sum_n \varphi(n) r_n$
diverges, where $\varphi$ is Euler's totient function. Hence for
almost all $\alpha$ we have the dichotomy that $\Rio$ is empty or
the entire circle depending on the convergence or divergence of
this series. However, for a given rotation number it is not clear
whether it belongs to this full measure set and hence if the
divergence of the series is the condition which determines the
dichotomy.}, exact dichotomy results were previously only known
for some self-similar sets equipped with the transformation
induced by the left shift on the coding, which have recently been
shown under the strong separation condition by Chang, Wu and Wu
\cite{Chang} and under the open set condition by Baker and Farmer
\cite{Baker}. And as previously mentioned, Hussein, Li, Simmons
and Wang \cite{Husseinetal} have shown exact dichotomy results
for some conformal and expanding systems.

Finally, we consider the set $\ear$ of eventually always
returning points.  In Theorem~\ref{theo:EARgen} we prove a result
in similar generality as Boshernitzan's Theorem on $\Rio$. For
the doubling map we give sufficient conditions for $\ear$ to be
of zero and full measure in Theorem~\ref{thm:EAR}. As for all
known results on $\E$ there is a range of shrinking rates not
allowing any conclusions on the size of $\ear$. It is an open
question whether one can prove a dichotomy condition on $\E$ or
$\ear$ for any system (see \cite[Question 28]{KKR}). 

In the next
section we state the main theorems and provide some intuition to
the results and their significance.

\section{Main results}

\subsection*{On the measure of $\Rio$ for a class of mixing interval maps}

Here we consider the case $X = [0,1]$. We will need the following
definition.

\begin{defn}[Decay of correlations for $L^1$ against $BV$] \label{defn:decay}
  Let $([0,1],T,\mu)$ denote a measure-preserving system. We say
  that correlations for the system decay as $p \colon \N \to \R$
  for $L^1$ against $BV$ (bounded variation), if
  \begin{equation*}
    \biggl| \int f \circ T^n g \, \mathrm{d} \mu - \int f \,
    \mathrm{d} \mu \int g \, \mathrm{d} \mu \biggr| \leq \lVert f
    \rVert_1 \lVert g \rVert_{BV} p (n)
  \end{equation*}
  holds for all $n$ and all functions $f$ and $g$ with $\lVert f
  \rVert_1\coloneqq \int |f| \, \mathrm{d} \mu <\infty $, $\lVert g
  \rVert_{BV}\coloneqq \var g + \sup |g|<\infty$, where $\var g$
  denotes the total variation of $g$. If $\sum_n p(n) < \infty$, then
  we say that the correlations are summable.
\end{defn}

Our first main result is the following.

\begin{maintheorem}\label{thm:riorate}
  Suppose that the system $([0,1],T,\mu)$ has exponential decay of
  correlations for $L^1$ against $BV$, and that $\mu$ is
  absolutely continuous with respect to Lebesgue measure with a
  density $h$ that is bounded away from zero and which belongs to
  $L^q$ for some $q > 1$.

  Let $r_n$ be a sequence of real numbers such that for any $c > 0$ we
  have
  \begin{equation} \label{eq:sequencecondition}
    \limsup_{N \to \infty} \sum_{n = c \log N}^N r_n = \infty
  \end{equation}
  Then $\mu (\Rio) = 1$.
\end{maintheorem}

\begin{remark}
  The condition \eqref{eq:sequencecondition} is strictly stronger than
  the condition $\sum^{\infty}_{n=1} r_n = \infty$. However,
  \eqref{eq:sequencecondition} is satisfied for many sequences, for
  instance if
  \[
  r_n \geq \frac{1}{n} \frac{1}{\prod^p_{j = 1} \log_j n}
  \]
  holds for some natural number $p$, where $\log_j$ denotes the
  logarithm iterated $j$ times.
  
  In particular, in the language of Boshernitzan and Pawelec the above
  result states that for any $p$ and for $\mu$-almost any $x\in [0,1]$
  \begin{equation*}
    \liminf_{n\geq 1} \biggl\{ n \biggl( \prod_{j = 1}^p \log_j n
    \biggr) d(T^n(x),x) \biggr\}=0.
  \end{equation*}
\end{remark}

\begin{remark}
  Systems which satisfies the assumptions of Theorem~\ref{thm:riorate}
  include some piecewise expanding maps \cite{Rychlik}, and quadratic
  maps with Benedicks--Carleson parameters as was proved by Young
  \cite{Young}. For piecewise expanding maps, the
  result of Hussein, Li, Simmons and Wang \cite{Husseinetal} is
  stronger than ours, since they only require that $\sum r_n =
  \infty$.  However, the result for quadratic maps is new.

  We remark also that recently Bylund has obtained results about
  the recurrence of the critical point in the quadratic family
  $f_a (x) = 1 - ax^2$ \cite{Bylund}. He gives a condition on
  $r_n$ which implies that the critical point $0$ belongs to
  $\Rio$ for a positive measure set of parameters. His condition is
  satisfied for instance for $r_n \geq \kappa / (n \log \log n)$,
  where $\kappa > 0$.
\end{remark}

Our method to prove Theorem~\ref{thm:riorate} allows us to also
consider more general measures than those that are absolutely
continuous with respect to Lebesgue measure. For such systems, it
is more natural to consider the set
\[
\newRio (M_n) = \{\, x : T^n (x) \in B(x, r_n(x)) \text{ for }
\infty \text{ many } n \in \mathbbm{N} \,\},
\]
where $r_n (x)$ is such that $\mu (B(x,r_n(x))) = M_n$. We will
prove the following theorem.

\begin{maintheorem} \label{thm:riorate2}
  Suppose that the system $([0,1],T,\mu)$ has exponential decay
  of correlations for $L^1$ against $BV$, and that there are
  constants $c$ and $s$ such that
  \[
  \mu (B(x,r)) \leq c r^s
  \]
  holds for all balls $B(x,r)$.

  Let $M_n$ be a sequence of real numbers such that for any $c > 0$ we
  have
  \begin{equation} \label{eq:sequencecondition2}
    \limsup_{N \to \infty} \sum_{n = c \log N}^N M_n = \infty
  \end{equation}
  Then $\mu(\newRio(M_n)) = 1$.
\end{maintheorem}

\begin{remark}
  When the invariant measure $\mu$ is not absolutely continuous with
  respect to Lebesgue measure, Theorem~\ref{thm:riorate2} is new also
  for piecewise expanding systems. For piecewise expanding systems,
  the result of Hussein, Li, Simmons and Wang \cite{Husseinetal} is
  only valid for measures that are absolutely continuous with respect
  to Lebesgue measure.
\end{remark}

Our next theorem concerns a sufficient condition for $\Rio$ and
$\newRio$ to be of zero measure.

\begin{maintheorem}\label{thm:RioMeasZero}
  Let $([0,1],T,\mu)$ denote a measure-preserving system for
  which correlations for $L^1$ against $BV$ are summable. Then
  \[
  \sum_{n=1}^{\infty} \int \mu(B(x,r_n)) \, \mathrm{d}
  \mu(x)<\infty \enskip\Rightarrow\enskip \mu(\Rio (r_n))=0
  \]
  and
  \[  
  \sum_{n=1}^{\infty} M_n < \infty \enskip\Rightarrow\enskip
  \mu(\newRio (r_n))=0.
  \]
\end{maintheorem}

\subsection*{Dichotomy results on the measure of $\Rio$ for some
  linear maps}

For some linear maps we are able to prove an exact dichotomy for
when $\Rio$ is of zero and full measure.

\begin{maintheorem}\label{thm:doublingdichotomy}
  Let $X = \mathbbm{T}^d = [0,1]^d$, $T(x)= A x \mod 1$, where
  $A$ is an integer matrix such that no eigenvalue is a root of
  unity.  Let $\mu$ denote the Lebesgue measure on $X$ and let
  $r_n$ be a sequence of non-negative numbers. Then
  \[
  \sum_{n=1}^{\infty} r_n^d < \infty \enskip\Rightarrow\enskip
  \mu(\Rio(r_n))=0.
  \]
  Moreover, if all eigenvalues of $A$ are outside the unit
  circle, then
  \[
  \sum_{n=1}^{\infty} r_n^d =\infty \enskip\Rightarrow\enskip
  \mu(\Rio(r_n))=1.
  \]
\end{maintheorem}

Note that the doubling map is a special case of the setting in
Theorem~\ref{thm:doublingdichotomy}. In fact, in dimension $d=1$,
Theorem~\ref{thm:doublingdichotomy} also follows from the results
obtained with different methods in \cite{Baker}, \cite{Chang} and
\cite{Husseinetal}.

\subsection*{Quantitative uniform recurrence results}

We turn to the set $\ear$ of eventually always returning
points. To state our result on speed of uniform recurrence in its
full generality we need the subsequent definition.

\begin{defn}
  Let $(X,d)$ be a metric space of finite diameter. For any $r >
  0$, let $N(r)$ denote the minimal number of balls of radius $r$
  that are needed to cover the space $X$. The number
  \[
  \udimb X = \limsup_{r \to 0} \frac{\log N(r)}{- \log r}
  \]
  is called the upper box dimension of $X$.
\end{defn}

Imitating the proof of Boshernitzan's Theorem~\ref{thm:Bos}, we
prove the following result in Section \ref{sec:EARgeneral}.

\begin{maintheorem} \label{theo:EARgen}
  Let $(X,T,\mu)$ be a measurable dynamical system with an
  invariant probability measure $\mu$ and a compatible metric $d$
  such that $(X,d)$ is a metric space of finite diameter and
  finite upper box dimension $\alpha > 0$. For every $\beta >
  \alpha$ and for $\mu$ almost every $x$ holds
  \[
  \lim_{m \to \infty} m^{\frac{1}{\beta}} \inf_{0 \leq k < m} d
  (x, T^k x) = 0.
  \]
\end{maintheorem}

By the same consideration as after Theorem~\ref{thm:Bos} this
statement can be reformulated for interval maps in the following
way: For any $\gamma<1$ and any $\kappa>0$ we have that
\begin{equation*} 
  r_n\geq \frac{\kappa}{n^{\gamma}} \text{ for all } n \enskip
  \Rightarrow \enskip \mu(\ear(r_n))=1.
\end{equation*}

In case of the doubling map we can improve this rate and our
results can be summarized as follows.

\begin{maintheorem}\label{thm:EAR}
  Let $X=[0,1]$, $T(x)=2x \mod 1$ and let $\mu$ denote the
  Lebesgue measure.
  \begin{enumerate}
  \item\label{thm:EARZero} Assume that $\lim_{m\to\infty} m
    r_m=0$. Then $\mu(\ear(r_n))=0$.
		
  \item\label{thm:EAROne} Suppose that $h$ is a function such
    that $h(n) \to \infty$ as $n \to \infty$, and let
    \begin{equation*}
      r_m=\frac{\log (m) h(m)}{m}.
    \end{equation*}
    Then $\mu(\ear(r_n))=1$.
  \end{enumerate}
\end{maintheorem}

\begin{remark}
  Theorem \ref{thm:EAR} holds true as well for transformations
  $T(x)=\beta x \mod 1$ for any $\beta \in \mathbbm{N}$, $\beta
  \geq 2$. The generalization is straightforward.
\end{remark}

We note that, to our knowledge, these are the first known results
on the measure of $\ear$.

\subsection{Intuition and motivation for the main results}

It is instructive to compare the type of statement presented in
Theorem~\ref{thm:riorate} (as well as Theorem~\ref{thm:Bos} and
\cite[Theorem~3.1]{Paw2017}) to the ones in
Theorem~\ref{thm:RioMeasZero} and \ref{thm:doublingdichotomy}.

In the context of $\Hio$, analogues of Theorem
\ref{thm:doublingdichotomy} are known as \emph{Dynamical
Borel--Cantelli lemmas} (DBCL's) and are known to hold for many
systems with nice mixing properties. One desirable feature of
this kind of result is that it gives an exact dichotomy for when
the set in question is of zero or full measure. Another advantage
to this type of statement is that it allows a great deal of
flexibility on the rate with which the targets are allowed to
shrink.

We do a short intermezzo here, clarifying the use of the word
\emph{shrinking} when referring to the targets. In DBCL's the
usual assumption is that the sum of the measure of the targets is
either finite or infinite. Hence shrinking in this context refers
to the \emph{measure} of the targets. (If the targets are nested,
then they are necessarily also shrinking in a geometric sence.)
This formulation also allows for more general targets than metric
balls when $X$ has more complex geometry than in our case.

The direct analogue of dynamical Borel--Cantelli lemmas for
$\Rio$ is to consider a convergence/divergence criteria for the
sum of the \emph{average} measure of the
targets. Theorem~\ref{thm:RioMeasZero} gives an example of the
convergence part of this type of condition. The averaging is
clearly necessitated by each sequence of targets being located in
a different region of the space $X$. Hence for any non-uniform
measure $\mu$, the measure of the targets depend on their
location. This also led to the introduction of the set $\newRio$. In principle Theorem~\ref{thm:doublingdichotomy} also
gives an example of a such condition, however, due to the
uniformity of the Lebesgue measure the averaging condition
collapses to a condition simply on the sum of the radii of the
targets.

In contrast, Theorem~\ref{thm:Bos} makes only an assumption on
the rate with which the radii $r_n$ go to zero, hence in this
context shrinking refers to the \emph{radii} of the balls around
$x$. Since Boshernitzan only assumes invariance of the measure,
no explicit connection between the radii of the balls and their
measure is assumed. However, the assumption in
Theorem~\ref{thm:Bos} that the space $X$ is $\sigma$-finite with
respect to the $\alpha$-dimensional Hausdorff measure, implies
that the set of points for which the local dimension of the
invariant measure $\mu$ is larger than $\alpha$, must be a small
set. Hence, there is implicitely present a weak assumption on the
connection between radii of most balls and their measure.

As for Boshernitzan's theorem, the result of Pawelec
\cite[Theorem~3.1]{Paw2017} and Theorem~\ref{thm:riorate} are
likewise formulated in terms of shrinking of the radii, however,
due to further assumptions on the invariant measure there exists
at least a partial connection between the radii and measures of
the targets in these cases.

It seems reasonable to expect that the dichotomy in
Theorem~\ref{thm:doublingdichotomy} holds also under the
assumptions of Theorem~\ref{thm:riorate}, but we are uncertain if
this is true. More generally, under sufficiently strong mixing
assumptions, one might expect that the divergence of the series
\[
\sum_{n=1}^{\infty} \int \mu (B(x,r_n)) \, \mathrm{d} \mu (x)
\]
implies that $\Rio$ has full $\mu$-measure.

The reason that we need the stronger assumption on $r_n$ in
Theorem~\ref{thm:riorate}, rather than the divergence of the
series above, is that the proof uses estimates on the correlation
of the sets $\{\, x : |x - T^n (x)| < r_n \,\}$.  Our estimates
on these correlations and the method of proof are not strong
enough to obtain $\mu (\Rio) = 1$ unless we impose extra
assumptions on the radii $r_n$.

\subsection{Structure of the paper}

We start by collecting several consequences of sufficiently fast
decay of correlation in Section \ref{sec:CorDecay}. These will
prove useful in the proofs of the main theorems on $\Rio$ in
Sections \ref{sec:proofA}--\ref{sec:proofC}. Finally, we consider
the set $\ear$ of eventually always returning points in Sections
\ref{sec:EARgeneral}--\ref{sec:ear}.

\section{Consequences of correlation decay}
\label{sec:CorDecay}

Throughout this section we set $X=[0,1]$ and $T\colon X\to X$. We
will deduce consequences from assumptions on the decay of
correlation for $L^1$ against $BV$. We start with the following
adaption of Lemma~3 in \cite{PerssonRams}. The statement is true
also for more general piecewise continuous functions $F$, but to
stay simple we formulate it for the kind of functions that we
will apply it to.

\begin{lemma} \label{lem:Decay1}
  Assume that $T\colon X\to X$ has summable decay $p(n)$ of
  correlations for $L^1$ against $BV$. Suppose that $F\colon
  [0,1]^2\to\R$ is the indicator function of an open or closed
  convex subset of\/ $[0,1]^2$.  Then
  \begin{equation*}
    \num{\int F(T^n x,x) \, \mathrm{d} \mu(x) - \iint F \,
      \mathrm{d} \mu \, \mathrm{d} \mu} \leq 3 p(n)
  \end{equation*}
\end{lemma}

\begin{proof}
  Let $Y \subset [0,1)^2$ be the convex subset such that $F (x,y)
    = \mathbbm{1}_{Y} (x,y)$. Take $\epsilon>0$ and fix $n$. Let
    $\hat{F}$ be a continuous function such that
  \begin{equation*}
    \num{\int F(T^n x,x) \, \mathrm{d} \mu(x)-\int \hat{F}(T^n
      x,x) \, \mathrm{d} \mu(x)}<\epsilon
  \end{equation*}
  and
  \begin{equation*}
    \num{\iint F \, \mathrm{d} \mu \mathrm{d} \mu - \iint \hat{F} \,
      \mathrm{d} \mu \mathrm{d} \mu}<\epsilon.
  \end{equation*}  
  We may choose $\hat{F}$ such that for all $x\in[0,1]$
  the function $f_x \colon y \mapsto \hat{F} (x,y)$ satisfies
  $\textup{var} f_x \leq 2$ and $\sup |f_x| \leq 1$. Indeed, put
  \[
  \hat{F}_k (x, y) = \left\{ \begin{array}{ll} 1 & \text{if }
    (x,y) \in Y, \\ \max \{ 0, 1 - k d( (x,y), Y) \} & \text{if }
    (x,y) \not \in Y, \end{array} \right.
  \]
  if $Y$ is closed and
  \[
  \hat{F}_k (x, y) = \left\{ \begin{array}{ll} 0 & \text{if }
    (x,y) \not \in Y, \\ \min \{ 1, k d( (x,y), \complement Y) \}
    & \text{if } (x,y) \in Y, \end{array} \right.
  \]
  if $Y$ is open, so that $\hat{F}_k$ is a piecewise linear
  approximation of $F$.

  If $Y$ is closed, then $\hat{F}_k \geq F$ and $\hat{F}_k$
  converges pointwise to $F$. If $Y$ is open, then $\hat{F}_k
  \leq F$ and $\hat{F}_k$ converges pointwise to $F$.  It follows
  that the functions $x \mapsto \hat{F}_k (T^n x,x)$ converges
  pointwise to $x \mapsto F (T^n x,x)$, and the convergence is
  monotone. Hence, by the monotone convergence theorem, we may
  take $k$ so large that
  \[
  \num{\int \hat{F}_k (T^n x,x) \, \mathrm{d} \mu(x)-\int F
  (T^n x,x) \, \mathrm{d} \mu(x)} < \varepsilon
  \]
  and
  \begin{equation*}
  	\num{\iint F \, \mathrm{d} \mu \mathrm{d} \mu - \iint  \hat{F}_k \,
  		\mathrm{d} \mu \mathrm{d} \mu}<\epsilon.
  \end{equation*}
  Let $\hat{F} = \hat{F}_k$ for such a $k$.
  
  Let $I_k = \bigl[ a_k, a_{k+1} \bigr)$, $k=0,\dots,m-1$ be a
    partition of $X$. Set
  \begin{equation*}
    G(x,y)=\sum_{k=0}^{m-1} \hat{F} (a_k ,y) \een_{I_k}(x),
  \end{equation*}
  where $\een_{I_k}$ denotes the characteristic function on
  $I_k$. Since $\hat{F}$ is continuous we may choose a partition
  so that
  \begin{equation*}
    |\hat{F} (x,y)-G(x,y)|<\epsilon
  \end{equation*}
  and hence
  \begin{equation*}
    \num{\int \hat{F}(T^n x,x) \, \mathrm{d} \mu(x)-\int G(T^n
      x,x) \, \mathrm{d} \mu(x)}<\epsilon.
  \end{equation*}
  The second integral can be rewritten as
  \begin{align*}
    \int G(T^n x,x) \, \mathrm{d} \mu(x)&=\sum_{k=0}^{m-1} \int
    \hat{F} (a_k ,x ) \een_{I_k}(T^n x) \, \mathrm{d} \mu(x).
  \end{align*}
  For the integral on the right hand side we may rewrite
  \begin{align*}
    \biggl| \int \hat{F} (a_k, x) \een_{I_k}(T^n x) \, \mathrm{d}
    \mu(x) - &\int \hat{F} (a_k, x) \,\mathrm{d}
    \mu(x)\int\een_{I_k}(x)\,\mathrm{d} \mu(x) \biggr|
    \\ &\leq\norm{\een_{I_k}}_1 \lVert \hat{F} (a_k,x)
    \rVert_{BV}p(n)\\ &\leq \mu(I_k) 3 p(n).
    \end{align*}
  Now summing over $k$ on both sides and using that $\sum_k
  \mu(I_k)= 1$ we get
  \begin{align*}
    \biggl| \int G(T^n x,x)\,\mathrm{d} \mu(x)-\int
    \sum_{k=0}^{m-1} \hat{F} (a_k, x) \mu(I_k)\, \mathrm{d}
    \mu(x) \biggr| \leq 3 p(n)
  \end{align*}
  and
  \begin{align*}
    \biggl| \int \hat{F} (T^n x,x) \, \mathrm{d} \mu(x)-\int
    \sum_{k=0}^{m-1} \hat{F} (a_k, x) \mu(I_k) \, \mathrm{d}
    \mu(x) \biggr| \leq \epsilon + 3 p(n).
  \end{align*}
  Hence
  \begin{align*}
    \biggl| \int F (T^n x,x) \, \mathrm{d} \mu(x)-\int
    \sum_{k=0}^{m-1} \hat{F} (a_k, x) \mu(I_k) \, \mathrm{d}
    \mu(x) \biggr| \leq 2\epsilon + 3 p(n)
  \end{align*}
  and
  \begin{multline*}
    \biggl| \int F (T^n x,x) \, \mathrm{d} \mu(x) - \iint F \,
    \mathrm{d} \mu \mathrm{d} \mu \biggr| \leq \biggl| \iint F \,
    \mathrm{d} \mu \mathrm{d} \mu - \iint \hat{F} \, \mathrm{d}
    \mu \mathrm{d} \mu \biggr| + \\ + \biggl| \iint \hat{F} \,
    \mathrm{d} \mu \mathrm{d} \mu -\int \sum_{k=0}^{m-1} \hat{F}
    (a_k, x) \mu(I_k) \, \mathrm{d} \mu(x) \biggr| + 2\epsilon +
    3 p(n).
  \end{multline*}
    
  Finally, by letting $\epsilon\to 0$ (i.e. $m\to\infty$), the
  sum converges to the integral $\iint \hat{F} \, \mathrm{d} \mu
  \mathrm{d} \mu$ and we get
  \[
  \biggl| \int F(T^n x,x)\, \mathrm{d} \mu(x)-\iint F (y,x) \,
  \mathrm{d} \mu(y) \mathrm{d} \mu(x) \biggr| \leq 3
  p(n). \qedhere
  \]
\end{proof}

In order to find correlation estimates in
Section~\ref{sec:proofA}, we need a lemma similar to
Lemma~\ref{lem:Decay1}, but for functions of three
variables. This is provided by the following lemma (with $M_n=\mu(B(x,r_n(x)))$ for every $n \in \mathbbm{N}$ as before).

\begin{lemma} \label{lem:3variables}
  Suppose that $(X,T,\mu)$ has exponential decay of correlations
  for $L^1$ against $BV$ (with $p(n)=Ce^{-\tau n}$ in Definition \ref{defn:decay}) and that there are constants $c$ and $s$
  such that
  \[
  \mu (B(x,r)) \leq c r^s
  \]
  holds for any ball $B(x,r)$.
  
  There is a constant $D$ such that for all $m,n \in
  \mathbbm{N}$, and for $F$ defined by
  \[
  F (x,y,z) = \left\{ \begin{array}{ll} 1 & \text{if } x \in B(z,
    r_{n+m} (z)) \text{ and } y \in B(z, r_n(z)), \\ 0 &
    \text{otherwise,} \end{array} \right.
  \]
  we have
  \begin{multline*}
    \int F(T^{n+m} (x), T^n (x), x) \, \mathrm{d} \mu (x) \leq (
    1 + 3C e^{-\frac{\tau}{2} n}) \iiint F \, \mathrm{d} \mu \,
    \mathrm{d} \mu \, \mathrm{d} \mu \\ + D (M_n e^{- \frac{s
        \tau}{2} n} + M_{m+n} (e^{- \frac{s \tau}{2} n} +
    e^{-\tau m}) + e^{- \frac{s \tau}{2} n}).
  \end{multline*}
\end{lemma}

\begin{proof}

Consider the integral $\int F(T^{n+m} x, T^n x, x) \, \mathrm{d}
\mu (x)$. We first approximate $F$ by partitioning $[0,1]$ into
$e^{\tau n /2}$ subintervals of equal length. The indicator
function of the $k$-th interval is $G_k$. Let $z_k$ be the mid
point of the $k$-th interval.

We write
\[
F (x,y,z) \leq \tilde{F} (x,y,z) := \sum_k F_k (x,y) G_k (z),
\]
where $F_k$ is chosen such that the above inequality is true, and
making the approximation close to as good as possible. More
precisely, we choose $F_k$ to be the indicator function of a
rectangle $\tilde{A}_{z_k}$, defined as follows.

Consider for fixed $z$ the sets $A_z = \{\, (x,y) : F (x,y,z) = 1
\, \}$, which is a rectangle of size $2r_{n+m} \times 2r_n$. By
expanding $A_{z_k}$ to a rectangle of size $2 (r_{n+m} + e^{-\tau
  n/2}) \times 2 (r_n + e^{-\tau n/2})$ and with the same centre,
we obtain the rectangle $\tilde{A}_{z_k} = I_k \times J_k$. By
the construction, we have that $F \leq \tilde{F}$.

Since $F_k$ is the indicator function of a rectangle
$\tilde{A}_{z_k} = I_k \times J_k$, we get by decay of
correlations that
\begin{align}
  \nonumber
  \int F_k (T^m x,x) \, \mathrm{d} \mu (x) & = \int
  \mathbbm{1}_{I_k} (T^m x) \mathbbm{1}_{J_k} (x) \, \mathrm{d}
  \mu (x) \\ & \leq \mu (I_k) (\mu (J_k) + 3 C e^{-\tau m}).
   \label{eq:Fkestimate}
\end{align}
and we get by the assumption $\mu (B(x,r)) \leq c r^s$ that
\[
\mu (I_k) \leq M_{n+m} + 2 c e^{-\frac{\tau s}{2} n} \qquad
\text{and} \qquad \mu (J_k) \leq M_n + 2 c e^{-\frac{\tau s}{2}
  n},
\]
which combined with \eqref{eq:Fkestimate}, implies that
\begin{equation} \label{eq:Fkestimate2}
  \int F_k (T^m x,x) \, \mathrm{d} \mu (x) \leq K_{m,n}
\end{equation}
where
\[
  K_{m,n} := (M_{n+m} + 2 c e^{-\frac{\tau s}{2} n}) (M_n + 2 c
  e^{-\frac{\tau s}{2} n} + 3C e^{-\tau m}).
\]

Using decay of correlations and \eqref{eq:Fkestimate2}, we now get
\begin{align*}
  \int &F(T^{n+m} x, T^n x, x) \, \mathrm{d} \mu (x) \\ & \leq
  \sum_k \int F_k(T^{n+m} x, T^n x) G_k (x) \, \mathrm{d} \mu (x)
  \\ & \leq \sum_k \int F_k(T^{m} x, x) \, \mathrm{d} \mu (x)
  \biggl( \int G_k \mathrm{d} \mu + 3 C e^{-\tau n} \biggr) \\ &
  \leq \sum_k K_{m,n} \biggl( \int G_k \mathrm{d} \mu + 3 C
  e^{-\tau n} \biggr) \\ &\leq K_{m,n} \biggl(1 + \sum_k
  3 C e^{-\tau n} \biggr),
\end{align*}
where in the last step, we used that $\sum_k \int G_k \,
\mathrm{d} \mu = 1$. Using that the sum over $k$ has $e^{\tau n /
  2}$ terms, we get that

\begin{align*}
  & \int F(T^{n+m} x, T^n x, x) \, \mathrm{d} \mu (x) \\ \leq & 
  (M_{n+m} + 2 c e^{-\frac{\tau s}{2} n}) (M_n + 2 c
  e^{-\frac{\tau s}{2} n} + 3C e^{-\tau m}) (1 + 3C
  e^{-\frac{\tau}{2} n}) \\ \leq &  M_n M_{n+m} ( 1 + 3C
  e^{-\frac{\tau}{2} n}) + D (M_n e^{- \frac{s \tau}{2} n} +
  M_{m+n} (e^{- \frac{s \tau}{2} n} + e^{-\tau m}) + e^{- \frac{s
      \tau}{2} n}).
\end{align*}
Since $\iiint F \, \mathrm{d}\mu\mathrm{d}\mu\mathrm{d}\mu = M_n
M_{m+n}$, the lemma follows.
\end{proof}

\section{Proof of Theorems~\ref{thm:riorate} and \ref{thm:riorate2}} \label{sec:proofA}
\subsection{Correlation estimates}

We are going to first prove Theorem~\ref{thm:riorate2} and then
conclude Theorem~\ref{thm:riorate} from
Theorem~\ref{thm:riorate2}.  To prove Theorem~\ref{thm:riorate2},
we suppose that a sequence $(M_n)_{n=1}^\infty$ is given, and we
define $r_n (x)$ such that $\mu (B(x,r_n(x))) = M_n$.

Let $\hat{E}_n = \{\, x : T^n (x) \in B(x,r_n(x)) \,\}$. Here, we
state and prove some estimates on the measure of $\hat{E}_n$ that
will be needed in the proof of Theorem~\ref{thm:riorate2}.

\begin{lemma} \label{lem:measureofEn}
  There is a constant $C>0$ such that
  \[
  M_n - C e^{-\tau n} \leq \mu (\hat{E}_n) \leq M_n + C e^{-\tau
    n}.
  \]
\end{lemma}

\begin{proof}
  Put
  \[
  F_n (x,y) = \left\{ \begin{array}{ll} 1 & \text{if } x \in B(y,
    r_n(y)), \\ 0 & \text{otherwise}. \end{array} \right.
  \]
  Then
  \[
  \mu (\hat{E}_n) = \int F_n (T^n (x), x) \, \mathrm{d} \mu (x).
  \]
  Consequently, we have by Lemma~\ref{lem:Decay1} that
  \[
  \iint F_n \, \mathrm{d} \mu \mathrm{d} \mu - C e^{-\tau n} \leq
  \mu (\hat{E}_n) \leq \iint F_n \, \mathrm{d} \mu \mathrm{d} \mu
  + C e^{-\tau n}
  \]
  By the definition of $F_n$ and by Fubini's theorem, we have
  \[
  \iint F_n \, \mathrm{d} \mu \mathrm{d} \mu = M_n,
  \]
  which proves the lemma.
\end{proof}

The following lemma is a direct consequence of
Lemma~\ref{lem:3variables}.

\begin{lemma} \label{lem:correlations}
  We have the following correlation estimate.
  \begin{multline}
    \mu (\hat{E}_{n} \cap \hat{E}_{n+m}) \leq (1 +3C e^{-\frac{
        \tau}{2} n} ) M_n M_{n+m} \\ + D (M_n e^{- \frac{s
        \tau}{2} n} + M_{m+n} (e^{- \frac{s \tau}{2} n} +
    e^{-\tau m}) + e^{- \frac{s \tau}{2} n}).
    \label{eq:correlation2}
  \end{multline}
\end{lemma}

We will use the correlation estimate from
Lemma~\ref{lem:correlations} to apply the following inequality by
Chung and Erd\H{o}s.

\begin{lemma}[The Chung--Erd\H{o}s inequality {\cite[Lemma]{ChungErdos}}] \label{lem:chungerdos}
  For measurable sets $A_1, \ldots, A_n$ holds
  \begin{equation}
    \mu (A_1 \cup \ldots \cup A_n) \geq \frac{ \Bigl(
      \sum_{j=1}^n \mu (A_j) \Bigr)^2 }{ \sum_{j,k = 1}^n \mu
      (A_j \cap A_k) }.
  \end{equation}
\end{lemma}

\subsection{Proof of Theorems~\ref{thm:riorate} and \ref{thm:riorate2}}

In this subsection we first give the proof of
Theorem~\ref{thm:riorate2} and later we conclude
Theorem~\ref{thm:riorate}. The proof is based on the
Chung--Erd\H{o}s inequality, Lemma~\ref{lem:chungerdos}.

\subsubsection{Proof of Theorem~\ref{thm:riorate2}}

We let
\[
U_N = \bigcup_{j \in I_N} \hat{E}_j
\]
where
\[
I_N = \{\, j : \frac{2}{\tau s} \log N \leq j \leq N \, \}.
\]

We note that $U_N$ is defined so that $\newRio = \limsup_{N\to
  \infty} U_N$, and to prove that $\newRio$ has large measure, we
will prove that $U_N$ has large measure.  In the union which
defines $U_N$ we consider only set $\hat{E}_j$ with $j \in
I_N$. By introducing the set $I_N$, we get better correlation
control. This has the effect that we need to require that
\[
\lim_{N \to \infty} \sum_{n = c \log N}^N M_n = \infty
\]
in order to prove that $U_N$ has large measure, which is the
reason that we cannot only assume that $\sum M_n$ is divergent.

Let
\[
S_N = \sum_{j \in I_N} \mu (\hat{E}_j) \qquad \text{and} \qquad
\sigma_N = \sum_{j \in I_N} M_j.
\]
By Lemma~\ref{lem:measureofEn}
\[
\sum_{j \in I_N} (M_j - C e^{- \tau j}) \leq S_N \leq \sum_{j \in
  I_N} (M_j + C e^{- \tau j}),
\]
so that
\[
\sigma_N - c_1 \leq S_N \leq \sigma_N + c_1,
\]
for some constant $c_1$.

We let
\[
C_N = \sum_{j,k \in I_N} \mu (\hat{E}_j \cap \hat{E}_k),
\]
and by Lemma~\ref{lem:correlations} we have that
\begin{align*}
  C_N &= S_N + 2 \sum_{\substack{j,k \in I_N \\ j > k}} \mu (\hat{E}_j
  \cap \hat{E}_k) \\ &\leq S_N + 2 \sum_{\substack{j,k \in I_N \\ j >
      k}} (1 + 3Ce^{-\frac{\tau}{2} k}) M_j M_k + R_N \\ & \leq S_N +
  (1 + 3CN^{-1/s}) \sigma_N^2 + R_N,
\end{align*}
where
\[
  R_N = 2 D \sum_{\substack{j,k \in I_N \\ j > k}} (M_k e^{-
    \frac{s \tau}{2} k} + M_{j} (e^{- \frac{s \tau}{2} k} +
  e^{-\tau (j-k)}) + e^{- \frac{s \tau}{2} k}) \bigr).
\]
We will prove that $R_N$ is bounded, and split $R_N$ into four
sums in a natural way.

For the first sum, we have
\[
\sum_{\substack{j,k \in I_N \\ j > k}} M_k e^{- \frac{s \tau}{2} k}
\leq \sum_{j = \frac{2}{s \tau} \log N}^N \sum_{k = \frac{2}{s \tau}
  \log N}^{j-1} e^{-\frac{s \tau}{2} k} \leq \sum_{j = \frac{2}{s
    \tau} \log N}^N c_2 e^{- \log N} \leq c_2.
\]
The same kind of estimate works for the second sum as well and
yields
\[
\sum_{\substack{j,k \in I_N \\ j > k}} M_j e^{- \frac{s \tau}{2} k}
\leq \sum_{j = \frac{2}{s \tau} \log N}^N \sum_{k = \frac{2}{s \tau}
  \log N}^{j-1} e^{-\frac{s \tau}{2} k} \leq \sum_{j = \frac{2}{s
    \tau} \log N}^N c_2 e^{- \log N} \leq c_2.
\]

For the third sum, we have
\begin{align*}
\sum_{\substack{j,k \in I_N \\ j > k}} M_j e^{-\tau (j - k)} & = \sum_{j
  = \frac{2}{s \tau} \log N}^N M_j \sum_{k = \frac{2}{s \tau} \log
  N}^{j-1} e^{-\tau (j - k)} \\
& \leq \sum_{j = \frac{2}{s \tau} \log N}^N
c_3 M_j = c_3 \sigma_N.
\end{align*}

Finally, the fourth sum is estimated as the first two by
\[
\sum_{\substack{j,k \in I_N \\ j > k}} e^{-\frac{s \tau}{2} k} \leq
c_2.
\]
Hence we have $R_N \leq 2D (3 c_2 + c_3 \sigma_N)$ and
\begin{align*}
  C_N & \leq S_N + (1 + 3CN^{-1/s}) \sigma_N^2 + 2D (3 c_2 + c_3 \sigma_N)
  \\ & \leq \sigma_N + c_1 + (1 + 3CN^{-1/s}) \sigma_N^2 + 2D (3 c_2 + c_3
  \sigma_N) \\
  & = (1 + 3CN^{-1/s}) \sigma_N^2 + c_4 (\sigma_N + 1).
\end{align*}

We now use the Chung--Erd\H{o}s inequality and conclude that
\[
\mu (U_N) \geq \frac{S_N^2}{C_N} \geq \frac{(\sigma_N - c_1)^2}{(1 +
  3CN^{-1/s}) \sigma_N^2 + c_4 (\sigma_N + 1)}.
\]
Since $\limsup_{N \to \infty} \sigma_N = \infty$ we conclude that
\[
\limsup_{N \to \infty} \mu(U_N) \geq 1,
\]
and hence that $\limsup_{N \to \infty} \mu (U_N) = 1$.  It follows
that $\mu (\limsup U_N) = 1$ and since $\limsup U_N = \limsup
\hat{E}_j$ we have proved that $\mu (\newRio) = 1$. This proves
Theorem~\ref{thm:riorate2}.

\subsubsection{Proof of Theorem~\ref{thm:riorate}}

We now conclude Theorem~\ref{thm:riorate} from
Theorem~\ref{thm:riorate2}. Since $\mu$ has density in $L^q$ for $q >
1$, it follows by H\"{o}lder's inequality that we may take $s = 1 -
\frac{1}{q}$. Hence the assumptions of Theorem~\ref{thm:riorate2} are
satisfied and we may conclude that for any sequence $M_n$ which
satisfies \eqref{eq:sequencecondition2} and for almost every $x$, we
have $T^n (x) \in B(x,r_n(x))$ for infinitely many $n$, where $r_n$ is
such that $\mu (B(x,r_n)) =M_n$. Since $\mu$ is absolutely continuous
with respect to Lebesgue measure with a density which is bounded away
from zero, this immediately implies that for almost all $x$ we have $d
(x, T^n(x)) < r_n$ for infinitely many $n$, provided the sequence
$r_n$ satisfies \eqref{eq:sequencecondition}.  This proves
Theorem~\ref{thm:riorate}.

\section{Proof of Theorem~\ref{thm:RioMeasZero}} \label{sec:RioMeasZero}

Given a sequence $(r_n)_{n=1}^\infty$ of non-negative numbers, we
will use the notation
\begin{align*}
  E_n:=\{\, x\in X:T^n(x)\in B(x,r_n) \,\}
\end{align*}
so that
\begin{align}
  \label{Rio}
  \Rio &=\bigcap_{k=1}^{\infty}\bigcup_{n=k}^{\infty} E_n .
\end{align}
Note that from \eqref{Rio} we get
\begin{align*}
  \mu(\Rio)\leq \lim_{k\to\infty}\sum_{n=k}^{\infty}\mu(E_n).
\end{align*}
Hence, if for fixed $k$ the sum converges we get $\mu(\Rio)=0$,
so we are interested in estimating the measure $\mu(E_n)$. For
that purpose, we will apply Lemma~\ref{lem:Decay1} on the
function $F_n$ defined by
\begin{equation*}
F_n(x,y)=\begin{cases}
1 & \text{if } x\in B(y,r_n),\\
0 & \text{otherwise.}
\end{cases}
\end{equation*}
This allows us to write 
\begin{align}\label{Enmeasure}
  \mu(E_n) = \int \een_{{E_n}} \mathrm{d}\mu &=\int F_n(T^n
  x,x)\, \mathrm{d}\mu(x).
\end{align}
By Lemma \ref{lem:Decay1} we then have
\begin{align}
  \num{\mu(E_n)-\iint F_n(y,x)\, \mathrm{d} \mu(y) \mathrm{d}
    \mu(x)}\leq 3p(n)
\end{align}
which turns into
\begin{align}\label{inftyeq}
  \num{\mu(E_n)-\int \mu(B(x,r_n)) \,\mathrm{d} \mu(x)}\leq
  3p(n)
\end{align}
and finally
\begin{align}
  \mu(E_n)\leq \int \mu(B(x,r_n)) \,\mathrm{d} \mu(x)+3p(n)
\end{align}
The term $p(n)$ is summable by assumption, hence we see that if
\begin{align*}
  \sum_{n=1}^{\infty}\int \mu(B(x,r_n)) \, \mathrm{d} \mu (x) <
  \infty
\end{align*}
then $\sum \mu(E_n)<\infty$ and consequently $\mu(\Rio)=0$. This
proves Theorem~\ref{thm:RioMeasZero}.

\section{Proof of Theorem~\ref{thm:doublingdichotomy}} \label{sec:proofC}

\subsection{The one dimensional case}

In this section we prove Theorem~\ref{thm:doublingdichotomy} in
the one dimensional case when $T\colon [0,1]\to [0,1]$ and $Tx=a
x\mod 1$, where $a$ is an integer with $|a| > 1$. We do this
since the proof in this case is simpler. In
Section~\ref{sec:higherdim} we give the proof of the higher
dimensional case. The higher dimensional case is similar to the
one dimensional case, but has some extra complications that are
not present in the one dimensional case.

We let $\mu$ denote the Lebesgue measure on $X=[0,1]$, which is a
$T$ invariant measure. In this case $\mu(B(x,r_n))=2r_n$. Note
that, in contrast to the general case of
Section~\ref{sec:RioMeasZero}, the right hand side is independent
of $x$.

The proof of the theorem will rely on an application of the
following lemma with $H = 1$. (See for instance Harman
\cite[Lemma~2.3]{Harman}, or conclude it yourself from the
Chung--Erd\H{o}s inequality.) The special case with $H = 1$ is
the Erd\H{o}s--Renyi formulation of the Borel--Cantelli lemma
\cite{ErdosRenyi}.

\begin{lemma} \label{lem:harman}
  Let $H > 0$. If $A_j$ are sets such that
  \begin{equation}\label{eq:divergentsum}
    \sum_{n=1}^{\infty} \mu(A_n)=\infty
  \end{equation}
  and
  \begin{equation} \label{eq:harman}
    \liminf_{N \to \infty} \frac{ \displaystyle \sum_{1 \leq i <
        j \leq N} \bigl( \mu (A_i \cap A_j) - H \mu (A_i) \mu
      (A_j) \bigr)}{\displaystyle \Biggl( \sum_{i=1}^N \mu (A_i)
      \Biggr)^2} \leq 0,
  \end{equation}
  then $\mu (\limsup A_j) \geq \frac{1}{H}$.
\end{lemma}

The strategy is to rewrite the quantity in the numerator of
\eqref{eq:harman} using Fourier series. The following two lemmas
will be helpful.

\begin{lemma} \label{lem:fouriercoeff}
  Let $f$ be a function of bounded variation on $[0,1]$. Let
  $f\sim \sum_{n\in\Z} c_n e^{2\pi i n x}$ be the Fourier series
  of the 1-periodic extension of $f$ to $\R$. Then $\num{c_n}
  \leq \frac{\var f}{2 \pi \num{n}}$ for any $n\neq 0$.
\end{lemma}

\begin{proof}
  Using Stieltjes integration, and integration by parts, we may write
  \begin{align*}
    c_n &= \int_0^1 e^{-i2\pi n x} f(x) \, \mathrm{d}x = \int_0^1
    \frac{-1}{i 2 \pi n} f(x) \, \mathrm{d} (e^{-i2\pi n x})
    \\ &= \int_0^1 \frac{1}{i 2 \pi n} e^{-i2 \pi nx} \,
    \mathrm{d}f (x).
  \end{align*}
  Hence $|c_n| \leq \frac{1}{2 \pi |n|} \var f$.
  
  For an elementary proof not using Stieltjes integrals, see Taibleson
  \cite{Taibleson}.
\end{proof}

\begin{lemma}\label{gcdlemma}
  Let $a,m,n\in\N$. Then $\gcd(a^m-1,a^n-1)=a^{\gcd(m,n)}-1$.
\end{lemma}

\begin{proof}
 Set $d=\gcd(a^m-1,a^n-1)$ and $k=\gcd(m,n)$. Hence the claim is
 that $d=a^k-1$. We will prove this by first showing that
 $a^k-1\mid d$ and afterwards that $d\mid a^k-1$.
 
 Since $k=\gcd(m,n)$ we have $k\mid m$ and $k\mid n$. Say $m=ks$,
 $n=kl$ for some $s,l\in \Z$. This means that we may write
 \begin{align*}
   a^m-1&=a^{ks}-1=\bigl(a^k
   \bigr)^s-(1)^s\\ a^n-1&=a^{kl}-1=\bigl(a^k \bigr)^l-(1)^l.
 \end{align*}
 We recall the general identity for $p,q,r\in\N$,
 \begin{align} \label{al:nt}
   (p^r-q^r)=(p-q)(p^{r-1}+p^{r-2}q+p^{r-3}q^2+\ldots +
   pq^{r-2}+q^{r-1})
 \end{align}
 which can be verified simply by multiplying the
 brackets. Applying this identity we get that
 \begin{align*}
   \bigl( a^k \bigr)^s-(1)^s &= \bigl( a^k-1 \bigr)
   \bigl(a^{k(s-1)}+a^{k(s-2)}+\ldots+a^k+1 \bigr)\\ \bigl( a^k
   \bigr)^l-(1)^l &= \bigl( a^k-1 \bigr) \bigl(
   a^{k(l-1)}+a^{k(l-2)}+\ldots+a^k+1 \bigr).
 \end{align*}
 Since all quantities in these two equations are integers, we can
 conclude that $a^k-1\mid a^n-1$ and $a^k-1\mid a^m-1$. Hence,
 $a^k-1\mid \gcd(a^m-1,a^n-1)=d$.
 
 To show $d\mid a^k-1$ we apply B\'ezout's Lemma on $k=\gcd(m,n)$
 which gives us $u,v \in \Z$ such that $um+vn=k$. On the one
 hand, we note that $u$ and $v$ cannot be both positive because
 then $k$ would be larger than $m$ and $n$. On the other hand,
 $u$ and $v$ cannot be both negative because then $k$ would be
 negative as well. Without loss of generality, we let $u>0$ and
 $v\leq 0$. Notice that if $v=0$, then $um=k$ which implies $u=1$
 and $k=m$. Clearly, $d=\gcd(a^m-1,a^n-1)$ divides $a^k-1$ in
 this case. So we examine the remaining situation $u>0$ and $v<
 0$. Then we use the identity (\ref{al:nt}) again to see that
 $d=\gcd(a^m-1,a^n-1)$ divides $a^{um}-1$ as well as
 $a^{-vn}-1$. Hence, $d$ divides $a^{um} - 1 - a^k (a^{-vn} -
 1)=a^k-1$.
 
 We conclude that $a^k-1=d$ and the lemma is proved.
\end{proof} 

We are now ready to prove Theorem~\ref{thm:doublingdichotomy} in the
one dimensional case.

\begin{proof}[Proof of Theorem~\ref{thm:doublingdichotomy} when $d = 1$]
  In the case of $\sum_{n=1}^{\infty} r_n<\infty$ the result will
  follow from the easy part of the Borel--Cantelli lemma.  In the
  case $\sum_{n=1}^{\infty} r_n=\infty$ the statement will follow
  from the special case of Lemma~\ref{lem:harman} with $H=1$.
  
  In our use of Lemma~\ref{lem:harman}, we let $A_n=E_n=\{\, x\in
  X:T^n(x)\in B(x,r_n) \,\}$ recalling that
  \begin{align}
    \label{Rio2}
    \Rio =\bigcap_{k=1}^{\infty}\bigcup_{n=k}^{\infty} E_n.
  \end{align}
  To analyse $\mu (E_n)$ and $\mu (E_n \cap E_m)$, we define the
  function
  \begin{equation*}
    G_n(x)=\begin{cases}
    1 & \text{ if } \num{x}<r_n\\
    0 & \text{ otherwise}.
    \end{cases}
  \end{equation*}
  Since $G_n(x)$ is a function on $\R/\Z$, we may periodically
  extend it to all of $\R$ and write it via its Fourier series,
  i.e.
  \begin{equation*}
    G_n(x)=\sum_{l\in\Z} c_{n,l} e^{2\pi i l x}.
  \end{equation*}
  The function $G_n (T^n x - x)$ is the indicator function of
  $E_n$. Hence, we have
  \[
  \mu (E_n) = \int G_n (T^n x - x) \, \mathrm{d} x = \sum_{l \in
    \mathbbm{Z}} c_{n,l} \int e^{- i 2 \pi l (a^n - 1) x} \,
  \mathrm{d} x.
  \]
  In the sum above, all integrals are zero, except for $l =
  0$. Hence we have
  \[
  \mu (E_n) = c_{n,0} = \int G_n (x) \, \mathrm{d} x = 2 r_n.
  \]
  
  It now follows that if $\sum r_n < \infty$ then $\sum \mu (E_n)
  < \infty$ and the easy part of the Borel--Cantelli lemma
  implies that $\mu (\Rio) = 0$. Of course, for $d=1$ this also
  follows directly from Theorem \ref{thm:RioMeasZero}.

  We assume from now on that $\sum r_n = \infty$.

  Using the Fourier series for $G_n$ we rewrite the quantity
  $\mu(E_m\cap E_n)$. We have
  \begin{align*}
    \mu(E_m\cap E_n) = \int \een_{E_m\cap E_n} \, \mathrm{d} \mu
    &=\int \een_{{E_m}}\een_{{E_n}}\, \mathrm{d} \mu\\ &=\int
    G_m(T^m x-x)G_n(T^n x-x) \, \mathrm{d} \mu.
  \end{align*}
  Hence we may write
  \begin{align}\label{Fouriersum}
    \int G_m&(T^mx-x) G_n(T^nx-x)\, \mathrm{d} \mu \nonumber\\
    &=\sum_{(k,l)\in\Z^2} c_{m,k}c_{n,l}\int e^{2\pi i (k(T^mx-x)+l(T^nx-x))} \, \mathrm{d} \mu\nonumber\\
    &=\sum_{(k,l)\in\Z^2} c_{m,k}c_{n,l}\int e^{2\pi i (k((a^m x\text{ mod }1)-x)+l((a^n x\text{ mod }1)-x))} \, \mathrm{d} \mu\nonumber\\
    &=\sum_{(k,l)\in\Z^2} c_{m,k}c_{n,l}\int e^{2\pi i (k(a^mx-x)+l(a^n x-x))} \, \mathrm{d} \mu\nonumber\\
    &=\sum_{(k,l)\in\Z^2} c_{m,k}c_{n,l}\int e^{2\pi i (k(a^m-1)+l(a^n-1))x} \, \mathrm{d} \mu\nonumber\\
    &=c_{m,0}c_{n,0}+\sum_{(k,l)\in\Z^2\backslash\tub{(0,0)}} c_{m,k}c_{n,l}\int e^{2\pi i (k(a^m-1)+l(a^n-1))x} \, \mathrm{d} \mu.
  \end{align}
  In the above equations we were allowed to ignore the
  $(\text{mod } 1)$ due to the periodicity of $e^{2\pi i k
    x}$. It is well known that
  \begin{equation*}
    \int e^{2\pi i (k(a^m-1)+l(a^n-1))x} \,\mathrm{d} \mu
    =\begin{cases} 1 & \text{if } k(a^m-1)+l(a^n-1)=0\\ 0 &
    \text{if } k(a^m-1)+l(a^n-1)\neq 0.
    \end{cases}
  \end{equation*}
  Hence we only get a contribution to the sum above when 
  \begin{equation}
    \label{gcdeq}
    k(a^m-1)+l(a^n-1)=0\enskip\iff\enskip -l=\frac{a^m-1}{a^n-1}
    k.
  \end{equation}
  In the following we look for the integer solutions $(k,l)$ to
  this equation. Generally we know, that given an equation
  $x=\frac{p}{q}y$ with $x,y,p,q \in \Z$, if $\frac{p}{q}$ are on
  lowest terms then the integer solutions to the equation are
  given by $(x,y)=(pj,qj)$, $j\in \Z$. Denote by
  $a_{(x,y)}:=\frac{a^x-1}{a^y-1}$. Lemma~\ref{gcdlemma} tells us
  that
  \begin{equation*}
    \frac{a_{(m,p)}}{a_{(n,p)}}, \quad p:=\gcd (m,n)
  \end{equation*}
  is on lowest terms and hence the integer pairs $(k,l)$ solving
  \eqref{gcdeq} are given by $\para{a_{(m,p)}j, a_{(n,p)}(-j)}$,
  $j\in\Z$. This means that the sum in \eqref{Fouriersum} may be
  rewritten as
  \begin{equation*}
    c_{m,0} c_{n,0} + \sum_{j\in\Z \setminus \{0\}} c_{m,a_{(m,p)}j}
    c_{n,a_{(n,p)}(-j)}.
  \end{equation*}
  Recall that $c_{m,0} = 2 r_m = \mu(E_m)$. Since $\var G_n \leq
  2$, we have by Lemma~\ref{lem:fouriercoeff} that $|c_{n,k}|
  \leq \frac{1}{\pi |k|}$. Using these estimates on the Fourier
  coefficients, we can now estimate the quantity $\mu(E_m\cap
  E_n)-\mu(E_m)\mu(E_n)$, namely,
  \begin{align*}
    \mu(E_m\cap E_n)-\mu(E_m)\mu(E_n) &= \sum_{j\in\Z\setminus
      \{0\}} c_{m,a_{(m,p)}j}c_{n,a_{(n,p)}(-j)}\\ &\leq
    \sum_{j\in\Z\setminus \{0\}}
    \num{c_{m,a_{(m,p)}j}}\num{c_{n,a_{(n,p)}(-j)}}\\ &\leq
    \frac{1}{\pi^2} \sum_{j\in\Z\setminus \{0\}}
    \frac{1}{\num{a_{(m,p)}j}}\frac{1}{\num{a_{(n,p)}(-j)}}\\ &=
    \frac{1}{\pi^2} \sum_{j\in\Z\setminus \{0\}} \frac{1}{\bigl|
      \frac{a^m-1}{a^p-1}j \bigr| \bigl| \frac{a^n-1}{a^p-1}(-j)
      \bigr| }\\ &= \frac{1}{\pi^2} \sum_{j\in\Z\setminus \{0\}}
    \frac{1}{\num{j}^2}\frac{(a^p-1)^2}{(a^m-1)(a^n-1)}\\ &<
    \frac{4}{\pi^2} \sum_{j\in\Z\setminus \{0\}}
    \frac{1}{\num{j}^2}\frac{a^{2p}}{a^{m+n}}\\ &\leq
    2a^{2p-(m+n)}.
  \end{align*}
  Inserting this in condition \eqref{eq:harman} of
  Lemma~\ref{lem:harman} with $H=1$ we get
  \begin{align*}
    \frac{\sum_{1\leq n<m\leq k}
      2a^{2p-(m+n)}}{\para{\sum_{n=1}^k
        \mu(E_n)}^2}&=\frac{2\sum_{m=1}^k \sum_{n=1}^{m-1}
      a^{2p-(m+n)}}{\para{\sum_{n=1}^k \mu(E_n)}^2}.
  \end{align*}
  By assumption we know that the denominator goes to infinity and
  we will show that the numerator converges for $k\to\infty$. We
  will do this by splitting the sum in the numerator as follows
  \begin{equation*}
    \sum_{m=1}^k\para{\sum_{n=1}^{\lfloor \frac{m}{2} \rfloor}
      a^{2p-(m+n)}+\sum_{n=\lceil \frac{m}{2}\rceil}^{m-1}
      a^{2p-(m+n)}}.
  \end{equation*} 
  We will use two different estimates. For the first sum we will
  use the trivial estimate $p\leq n$. For the second sum we will
  use that $m=px$, $n=py$, $x,y \in\N$ implies that
  $m-n=p(x-y)\geq p$ since $m>n$. Using this we get that
  \begin{align*}
    \sum_{1\leq n<m\leq k} a^{2p-(m+n)}&\leq
    \sum_{m=1}^k\para{\sum_{n=1}^{\lfloor \frac{m}{2} \rfloor}
      a^{2n-(m+n)}+\sum_{n=\lceil \frac{m}{2}\rceil}^{m-1}
      a^{2(m-n)-(m+n)}}\\ &=\sum_{m=1}^k\para{\sum_{n=1}^{\lfloor
        \frac{m}{2} \rfloor} a^{n-m}+\sum_{n=\lceil
        \frac{m}{2}\rceil}^{m-1} a^{m-3n}}\\ &\leq
    \sum_{m=1}^k\para{\frac{m}{2} a^{\frac{m}{2}-m}+ \frac{m}{2}
      a^{m-3\frac{m}{2}}}\\ &= \sum_{m=1}^k m
    a^{-\frac{m}{2}}\\ &\leq \sum_{m=1}^\infty m
    a^{-\frac{m}{2}}.
  \end{align*}
  This series converges and hence condition \eqref{eq:harman} is
  satisfied. The result then follows from Lemma~\ref{lem:harman}.
\end{proof}

\subsection{The case of general dimension} \label{sec:higherdim}

The proof follows the strategy employed for the one dimensional
case but with certain adaptations. The notation $E_n$ and $G_n$
remains unchanged.

We first take care of the measure zero case which is just a
simple adaptation. Written as Fourier series $G_n$ becomes
\begin{equation*}
  G_n(x)=\sum_{l\in\Z^d} c_{n,l} e^{2\pi i \ip{l}{x}}.
\end{equation*}
In the same way as in the one dimensional case, we have
\begin{align*}
  \mu (E_n) = \int G_n (T^n x - x) \, \mathrm{d} x &= \sum_{l \in
    \mathbbm{Z}^d} c_{n,l} \int e^{i 2 \pi \ip{l}{(A^n - I) x}}
  \, \mathrm{d} x \\ &= \sum_{l \in \mathbbm{Z}^d} c_{n,l} \int
  e^{i 2 \pi \ip{(A^n - I)^T l}{x}} \, \mathrm{d} x.
\end{align*}
Since $A$ has no eigenvalues that are roots of unity, the matrix
$(A^n - I)^T$ is an invertible integer matrix and $(A^n - I)^T l
= \bar{0}$ only if $l = \bar{0}$, where $\bar{0}$ denotes the
zero-vector in $d$ dimensions.  Hence, all integrals in the sum
above are zero, unless $l = \bar{0}$. It follows that
\[
\mu (E_n) = c_{n, \bar{0}} = \int G_n \, \mathrm{d} x = c_d
r_n^d,
\]
where $c_d$ is the volume of the $d$ dimensional unit ball.  Now,
if $\sum r_n^d < \infty$ then $\sum \mu (E_n) < \infty$ and the
easy part of the Borel--Cantelli lemma implies that $\mu (\Rio) =
0$.

We assume from now on that $\sum r_n^d = \infty$ and that all
eigenvalues of $A$ lie outside the unit circle. There is a number
$\lambda > 1$ such that all eigenvalues of $A$ have modulus
strictly larger than $\lambda$. In this case we approximate the
functions $G_n$ by $C^r$-functions. As a parameter in this
approximation, we choose $\varepsilon > 0$.  Let $f \in
C^r([0,1])$ be such that $f$ is monotone, $f(0) = 1$, $f(1)=0$
and $f'$ has compact support in $(0,1)$.  Put $f_n (t) = f
\bigl(\frac{t-r_n}{\varepsilon r_n} \bigr)$.

We approximate $G_n$ by
\begin{equation*}
  \widetilde{G}_n(x)=\begin{cases} 1 & \text{ if } \num{x}\leq
  r_n\\ f_n(|x|) & \text{ if } r_n<\num{x}\leq
  (1+\epsilon)r_n\\ 0 & \text{ if } (1+\epsilon)r_n<\num{x}
	\end{cases}
\end{equation*}
where $|x|$ denotes the length of the vector $|x|$.  Note that
$G_n\leq\widetilde{G}_n$.  Written as Fourier series
$\widetilde{G}_n$ becomes
\begin{equation*}
  \widetilde{G}_n(x)=\sum_{l\in\Z^d} \widetilde{c}_{n,l} e^{2\pi
    i \ip{l}{x}}.
\end{equation*}
Analogous to above, we get
\[
(1 + \varepsilon)^d \mu (E_n) \geq \int \widetilde{G}_n \,
\mathrm{d} x = \tilde{c}_{n,\bar{0}}.
\]
Since $\widetilde{G}_n$ is a $C^r$-function, a scaling argument
gives the estimate
\begin{equation} \label{eq:Fourierdecay}
  |\widetilde{c}_{n,l}| \leq \frac{C r_n^{-r}}{|l|^r},
\end{equation}
where $C$ is a uniform constant.

Without loss of generality we may assume that the sequence $r_n$
satisfies
\begin{equation} \label{eq:decayassumption}
  r_n \leq \frac{1}{n^2} \qquad \Rightarrow \qquad r_n = 0.
\end{equation}
We will prove that the set $\Rio$ has full measure under this
assumption. If this assumption is not satisfied, then we may
simply replace each $r_n$ which satisfies $r_n \leq 1/n^2$ by
$r_n = 0$. This does not change $\sum r_n^d = \infty$ and the
resulting set $\Rio$ is smaller but of full measure, so that the
set $\Rio$ for the original sequence $r_n$ is of full measure as
well.

Since $G_n\leq\widetilde{G}_n$ we have
\begin{align*}
  \mu(E_m\cap E_n)\leq \int \widetilde{G}_m(T^m
  x-x)\widetilde{G}_n(T^n x-x) \, \mathrm{d} \mu.
\end{align*}
This gives
\begin{multline}\label{Fouriersum2}
  \int \widetilde{G}_m(T^mx-x) \widetilde{G}_n(T^nx-x)\,
  \mathrm{d} \mu \\ = \widetilde{c}_{m,\bar{0}}
  \widetilde{c}_{n,\bar{0}}+\sum_{(k,l)\in\Z^d\times
    \Z^d\backslash\tub{(\bar{0},\bar{0})}} \widetilde{c}_{m,k}
  \widetilde{c}_{n,l}\int e^{2\pi i \ip{(A^m-I)^T k+(A^n-I)^T
      l}{x}} \, \mathrm{d} \mu.
\end{multline}

Analogue to the one dimensional case, we only get a contribution to the
sum above when
\begin{align*}
  (A^m-I)^T k + (A^n-I)^T l = 0 &\iff (A^m-I)^T k = - (A^n-I)^T l
  \\ &\iff ((A^T)^m - I) k = - ((A^T)^n - I) l, \nonumber
\end{align*}
where $l,k\in\Z^d$.  We will need the following lemma. We thank Victor
Ufnarovski for proving this lemma for us.

\begin{lemma} \label{lem:ufnarovski}
  Let $B$ be a square integer matrix such that no eigenvalue is a
  root of unity.  Let $p=\gcd (m,n)$. Then
  \begin{equation*}
    (B^m-I)k=(B^n-I)l, \quad l,k\in\Z^d
  \end{equation*}
  if and only if
  \begin{equation*}
    \begin{cases}
      k=(I+B^p+\ldots+B^{n-p})j\\
      l=(I+B^p+\ldots+B^{m-p})j
    \end{cases}
  \end{equation*}
  for some $j\in\Z^d$.
\end{lemma}

\begin{proof}
  Replacing $B^p$ by $B$, we may assume that $\gcd (m,n) = 1$. We are
  then to prove that $(B^m-I)k=(B^n-I)l$ holds if and only if
  \begin{equation} \label{eq:solutions}
    \begin{cases}
      k=(I+B+\ldots+B^{n-1})j\\
      l=(I+B+\ldots+B^{m-1})j
    \end{cases}
  \end{equation}
  for some $j\in\Z^d$.

  Since $(B^m - I)(I+B+\ldots+B^{n-1}) = (B^n -
  I)(I+B+\ldots+B^{m-1})$ it is clear that \eqref{eq:solutions} are
  solutions to the equation $(B^m-I)k=(B^n-I)l$. It remains to prove
  that these are the only solutions.

  We first prove that there are integer polynomials $u$ and $v$ such
  that
  \begin{equation} \label{eq:euclide}
  u (x) (1 + x + \ldots + x^{m-1}) + v (x) (1 + x + \ldots + x^{n-1})
  = 1.
  \end{equation}
  Let $Z$ be the set of pairs $(m,n)$ of natural numbers for which
  \eqref{eq:euclide} holds for some integer polynomials $u$ and
  $v$. Clearly, $(1,1) \in Z$.

  Suppose that $m > n$. Then $(m-n,n) \in Z$ implies that $(m,n) \in
  Z$. Similarly, if $n > m$, then $(m, n-m) \in Z$ implies that $(m,n)
  \in Z$.

  Since $\gcd(m,n) = 1$, we can repeatedly reduce the pair $(m,n)$ by
  replacing it with $(m-n, n)$ or $(m, n-m)$, and as in the Euclidean
  algorithm, this proceedure will eventually end up in the pair $(1,1)
  \in Z$. Hence $(m,n) \in Z$. This proves that there are integer
  polynomial $u$ and $v$ such that \eqref{eq:euclide} holds.
  
  From \eqref{eq:euclide}, we get
  \[
  u (B) (I + B + \ldots + B^{m-1}) + v(B) (I + B + \ldots + B^{n-1}) =
  I,
  \]
  and in particular
  \begin{equation} \label{eq:kequality}
  k = u(B) (I + B + \ldots + B^{m-1}) k + v(B) (I + B + \ldots +
  B^{n-1}) k
  \end{equation}
  for any vector $k$.

  Let $P$ be such that $P = (B^m - I) k = (B^n - I) l$ for some $k,l
  \in \mathbbm{Z}^d$. Then 
  \begin{align*}
    P &= (B - I) (I + B + \ldots + B^{m-1}) k \\ &= (B - I) (I +
    B + \ldots + B^{n-1}) l \\ &= (B^m - I)k = (B^n - I)l,
  \end{align*}
  and hence
  \[
  (I + B + \ldots + B^{m-1}) k = (I + B + \ldots + B^{n-1}) l.
  \]
  We have
  \[
  u (B) (I + B + \ldots + B^{m-1}) k = u(B) (I + B + \ldots + B^{n-1}) l.
  \]
  Using \eqref{eq:kequality} and the fact that $u(B), v(B)$ are
  polynomials in $B$, we get that
  \begin{align*}
    k &= u(B) (I + B + \ldots + B^{m-1}) k + v(B) (I + B + \ldots
    + B^{n-1}) k \\ &= u(B) (I + B + \ldots + B^{n-1}) l + v(B)
    (I + B + \ldots + B^{n-1}) k \\ &= (I + B + \ldots + B^{n-1})
    ( u(B) l + v(B) k)\\ &= (I + B + \ldots + B^{n-1}) j
  \end{align*}
  where $j = u(B) l + v(B) k$ is an integer vector. Similarly, we
  get
  \[
  l = (I + B + \ldots + B^{m-1}) j
  \]
  with the same $j$.
\end{proof}

In the following, set
\begin{align*}
  A_m:&=I+(A^T)^p+\ldots+(A^T)^{m-p},\\
  A_n:&=I+(A^T)^p+\ldots+(A^T)^{n-p},
\end{align*}
where $p$ always denotes $p = \gcd(m,n)$.

Lemma~\ref{lem:ufnarovski} tells us that the sum in \eqref{Fouriersum2}
may be rewritten as
\begin{equation*}
  \widetilde{c}_{m,\bar{0}} \widetilde{c}_{n,\bar{0}} +
  \sum_{j\in\Z^d \setminus \{\bar{0}\}} \widetilde{c}_{m,A_n j}
  \widetilde{c}_{n, - A_m j}
\end{equation*}
and as was noted above
\begin{equation*}
  \widetilde{c}_{n,\bar{0}} = \int \widetilde{G}_n(x)\,
  \mathrm{d} \mu \leq (1+\epsilon)^d\mu(E_n).
\end{equation*}
Hence, if we let $H = (1 + \varepsilon)^{2d}$, then
\[
\sum_{m,n= 1}^N \bigl( \mu(E_n \cap E_m) - H \mu (E_n) \mu (E_m)
\bigr) \leq \sum_{m,n= 1}^N \sum_{j\in\Z^d \setminus \{\bar{0}\}}
\widetilde{c}_{m,A_n j} \widetilde{c}_{n, - A_m j}.
\]

Using the estimate \eqref{eq:Fourierdecay} and the assumption
\eqref{eq:decayassumption} we have if $r_n, r_m \neq 0$ that 
\[
|\widetilde{c}_{m,A_n j} \widetilde{c}_{n,A_m j}| \leq C (r_n
r_m)^{-r} \frac{1}{| A_n j |^r | A_m j |^r} \leq C^2
n^{2r} m^{2r} \frac{1}{| A_n j |^r | A_m j |^r}.
\]
We may estimate that
\[
\num{A_n j}  \geq c \lambda^{n-p} \num{j } \qquad \text{and}
\qquad \num{A_m j} \geq c \lambda^{m-p} \num{j},
\]
for some uniform constant $c$.
Hence
\[
|\widetilde{c}_{m,A_n j} \widetilde{c}_{n,A_m j}| \leq C^2 c^{-2}
n^{2r} m^{2r} \lambda^{r(2p - m - n)} \num{j}^{-2r},
\]
an estimate which is certainly true also when either $r_n = 0$ or
$r_m = 0$, since then all the corresponding Fourier coefficients
are zero and $|\widetilde{c}_{m,A_n j} \widetilde{c}_{n,A_m j}| =
0$.

We conclude that if $r$ is sufficiently large, then
\begin{align*}
  \sum_{m,n= 1}^N \bigl( \mu(E_n \cap E_m) &- H \mu (E_n) \mu
  (E_m) \bigr)\\ &\leq \sum_{m,n= 1}^N \sum_{j\in\Z^d \setminus
    \{\bar{0}\}} C^2 c^{-2} n^{2r} m^{2r} \lambda^{r(2p - m - n)}
  \num{j }^{-2r} \\ & = C' \sum_{m,n= 1}^N n^{2r} m^{2r}
  \lambda^{r(2p - m - n)}.
\end{align*}
Just as in the one dimensional case, this is bounded as $N \to
\infty$. Thus, Lemma~\ref{lem:harman} implies that $\mu(\Rio)
\geq \frac{1}{H} = (1+\varepsilon)^{-2d}$. As $\varepsilon$ can
be taken as small as we wish, we can make $H$ arbitrarily close
to $1$ and we conclude that $\mu(\Rio)=1$.

\section{Proof of Theorem \ref{theo:EARgen}} \label{sec:EARgeneral}

In this Section we prove our very general result on quantitative
uniform recurrence.

\begin{proof}[Proof of Theorem \ref{theo:EARgen}]
  If $V$ is a measurable set and $t \in \mathbbm{N}$, then we let
  \[
  V (t) = \{\, x \in V : T^i (x) \not \in V \text{ for all }1\leq
  i \leq t \,\}.
  \]
  We have $\mu (V(t)) \leq 1/t$ because $T^{-i}(V(t))$, $1\leq i
  \leq t$, are disjoint sets of equal measure $\mu(V(t))$.
	
  Fix $\beta > \alpha$ and choose $\beta'$ with
  $\beta>\beta'>\alpha$.  The proof contains two parameters
  $\gamma > 0$ and $\theta > 1$ that will be chosen later. For
  each $m \in \mathbbm{N}$, we let $R_m = m^{-\gamma}$, and we
  take $p_m$ so that
  \[
  m = \frac{4 p_m^2}{R_m^{\beta'}} \qquad \Leftrightarrow \qquad
  p_m = \frac{1}{2} m^{\frac{1 - \beta' \gamma}{2}}.
  \]
  We require that $\gamma \in (1/\beta, 1/\beta')$ so that $p_m
  \to \infty$ with $m$.
	
  When $m$ is large enough, since $\udimb X < \beta' < \beta$,
  there is a cover $\{B_{m,i}\}_{i \in I_{m,0}}$ of $X$ by balls
  of diameter $|B_{m,i}|=R_m$ such that the number of balls is at
  most $R_m^{-\beta}$. We then have
  \begin{equation} \label{eq:coversum}
    \sum_i |B_{m,i}|^{\beta'} \leq R_m^{\beta' - \beta} < 1
  \end{equation}
  if $m$ is large enough.
        
  By Vitali's covering lemma which holds true in every metric
  space \cite[Theorem 1.2]{Heinonen}, there is $I_m \subset
  I_{m,0}$ such that the balls $B_{m,i}$, $i \in I_m$ are
  pairwise disjoint and such that $\{5 B_{m,i}\}_{i \in I_m}$
  covers $X$. (If $B$ is a ball, then $5B$ denotes the ball with
  same centre as $B$ and with radius $5$ times as large.)
	
  By replacing the balls in the cover $\{5 B_{m,i}\}_{i \in I_m}$
  by subsets, we can get a cover $\{V_{m,i}\}_{i \in I_m}$ of
  $X$, such that the sets $V_{m,i}$ are pairwise disjoint and
  such that
  \[
  1 \leq \frac{|V_{m,i}|}{R_m} \leq 5.
  \]
	
  Let $J_m = \{\, i \in I_m : 2 p_m \mu (V_{m,i}) \leq
  |B_{m,i}|^{\beta'} \,\}$. Then
  \begin{equation} \label{eq:measure1}
    \mu \biggl( \bigcup_{i \in J_m} V_{m,i} \biggr) = \sum_{i \in
      J_m} \mu (V_{m,i}) \leq \frac{1}{2p_m} \sum_i
    |B_{m,i}|^{\beta'} < \frac{1}{2 p_m},
  \end{equation}
  by \eqref{eq:coversum}.
	
  For $i \not \in J_m$ we have
  \begin{equation} \label{eq:measure2}
    \mu (V_{m,i} (m)) \leq \frac{1}{m} \leq \frac{R_m^{\beta'}}{4
      p_m^2} \leq \frac{2 p_m \mu (V_{m,i})}{4 p_m^2} =
    \frac{1}{2 p_m} \mu (V_{m,i}).
  \end{equation}
	
  Let
  \[
  G_m = \bigcup_{i \in J_m} V_{m,i} \cup \bigcup_{i \not \in J_m}
  V_{m,i} (m).
  \]
  By \eqref{eq:measure1} and \eqref{eq:measure2} we then have
  \[
  \mu (G_m) \leq \frac{1}{2 p_m} + \frac{1}{2 p_m} =
  \frac{1}{p_m}.
  \]
	
  Notice that if $x \in \complement G_m$, then there is an $i$
  and a $k$ with $1 \leq k \leq m$ such that $x, T^k x \in
  V_{m,i}$. Hence, $d (x, T^k x) \leq 5 R_m$.
	
  Put $m_j = j^\theta$. Then
  \[
  \sum_{j = 1}^\infty \mu (G_{m_j}) \leq \sum_{j=1}^\infty
  p_{m_j}^{-1} = \sum_{j=1}^\infty 2 j^{- \theta \frac{1 - \beta'
      \gamma}{2}},
  \]
  which is convergent provided $\theta \frac{1 - \beta'
    \gamma}{2} > 1$. Since $1 - \beta' \gamma > 0$, we can choose
  $\theta > 1$ sufficiently large so that the above series is
  convergent.
	
  It then follows by the Borel--Cantelli lemma that for such a
  choice of $\theta$, we have
  \[
  \mu (\limsup_{j \to \infty} G_{m_j}) = 0.
  \]
  Therefore, we have
  \[
  \mu (\liminf_{j \to \infty} \complement G_{m_j}) = 1.
  \]
  Let $F = \liminf_{j \to \infty} \complement G_{m_j}$. Whenever
  $x \in F$, there is a $j_0$ which depends on $x$ such that for
  any $j > j_0$, there is a $k \leq m_j = j^\theta$ with
  \[
  d (x, T^k x) \leq 5 R_{m_j} = 5 j^{- \theta \gamma}.
  \]
	
  Let $x \in F$ and suppose that $m > j_0^\theta$. There is then
  a $j$ such that $j^\theta < m \leq (j+1)^\theta$. There is
  therefore a $k \leq j^\theta < m$ such that
  \[
  d (x, T^k x) \leq 5 R_{m_j} = 5 j^{- \theta \gamma} = 5 \biggl(
  \frac{j+1}{j} \biggr)^{\theta \gamma} (j+1)^{-\theta \gamma}
  \leq 5 \cdot 2^{\theta \gamma} m^{-\gamma}.
  \]
  Consequently, for any large enough $m$, there is a $k < m$ with
  \[
  d(x, T^k x) \leq 5 \cdot 2^{\theta \gamma} m^{-\gamma}.
  \]
  Since $\gamma > 1/\beta$ the theorem follows from this statement.
\end{proof}

\section{Eventually always returning points for the doubling map} \label{sec:ear}

We consider the set of \emph{eventually always returning points}
defined by
\begin{align*}
  \ear&:=\Bigl\{\, x\in X: \exists\, n\in \N\, \forall\, m\geq n:
  \bigl\{T^k x \bigr\}_{k=1}^{m}\cap B(x,r_m)\neq \emptyset
  \,\Bigr\}\\ &=\bigcup_{n=1}^{\infty} \bigcap_{m=n}^{\infty}
  \bigcup_{k=1}^{m} \bigl\{\, x \in X : T^k x \in B(x,r_m)
  \,\bigr\}\\
\end{align*}
for the doubling map $T(x)=2x \mod 1$ on $X=[0,1]$. We write
\begin{align*}
  E_{k,m}:=& \bigl\{\, x\in X : T^k x \in B(x,r_m) \,
  \bigr\},\\ C_m:=&\bigcup_{k=1}^{m}
  E_{k,m},\\ A_n:=&\bigcap_{m=n}^{\infty} C_m.
\end{align*}
Since $A_{n} \subset A_{n+1}$ we have $\mu ( \ear ) = \lim_{n \to
  \infty} \mu(A_n)$.

\subsection{Sufficient condition for measure one}

\begin{prop} \label{prop:Ear}
  Let $X=[0,1]$, $T(x)=2x \mod 1$ and let $\mu$ denote the
  Lebesgue measure. We consider $r_m = \frac{\Delta_m h
    (\Delta_m)}{m}$ with $h(n) \to \infty$ as $n\to \infty$,
  $\Delta_m \to \infty$ as well as $\frac{m}{\Delta_m}\to \infty$
  and $\frac{m^{2+\sigma}}{\Delta_m} 2^{-\Delta_m} \leq 1$ for
  some $\sigma > 0$ for $m$ sufficiently large. Then
  $\mu(\ear)=1$.
\end{prop}

\begin{proof}
  Let $\varepsilon_m = \frac{1}{m^{1+\sigma}}$. 
  We define the function
  \begin{equation*}
    F_m(t)=\begin{cases}
    0 & \text{if } \abs{t} < r_m,\\
    1 & \text{otherwise.}
    \end{cases}
  \end{equation*}
  In order to describe the return of the point $x$ under $T^n$ we
  can use the map $F_m$ in the following way:
  \begin{equation*}
    G_{k,m}(x) \coloneqq F_m \bigl(T^kx - x \bigr)
    = \begin{cases} 0 & \text{if } \abs{T^kx-x} < r_m,\\ 1 &
      \text{otherwise.}
    \end{cases}
  \end{equation*}
  Then $G_{k,m}$ is the characteristic function of $\complement
  E_{k,m}$ using the notation from above. In that notation we
  also have
  \begin{align*}
    \mu \bigl( \complement C_m \bigr) = \mu \Biggl(
    \bigcap_{k=1}^m \complement E_{k,m} \Biggr) = \int
    \prod^m_{k=1} G_{k,m}(x) \; \dd x \leq \int
    \prod^{m/\Delta_m}_{k=1} G_{p_k,m}(x) \; \dd x
  \end{align*}
  with\footnote{For the sake of convenience we treat the numbers
  $p_k$ and $\frac{m}{\Delta_m}$ as integers avoiding the use of
  floor functions.} $p_k = k \cdot \Delta_m$. Thinking of $[0,1]$
  as the circle and identifying the endpoints of $[0,1]$, we note
  that $G_{p_1,m}(x)=F_m ( ( 2^{p_1}-1 )x )$ attains the value
  $1$ on $2^{p_1}-1$ many intervals of length
  $\frac{1-2r_m}{2^{p_1}-1}$, see Figure~\ref{fig:2xmod1fig}. On
  each of these intervals, the function $G_{p_2,m}$ takes the
  value $0$ on at least
  \begin{equation}
    \label{eq:numberIntervals}
    \frac{\frac{1-2r_m}{2^{p_1}-1} }{\frac{1}{2^{p_2}-1}} -2 =
    \biggl(1-2r_m -2 \cdot \frac{2^{p_1}-1}{2^{p_2}-1} \biggr)
    \cdot \frac{2^{p_2}-1}{2^{p_1}-1}
  \end{equation}
  many intervals of length $\frac{2r_m}{2^{p_2}-1}$.

  \begin{figure}
    \begin{center}
      \includegraphics{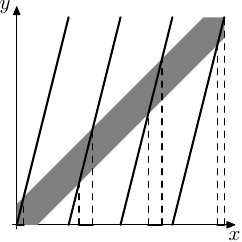} \qquad
      \includegraphics{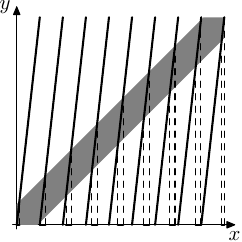}
    \end{center}
    \caption{The function $G_{k, m}$ is $1$ on the intervals
      marked on the $x$-axes. } \label{fig:2xmod1fig}
  \end{figure}
  
  For the rest of the proof, all inequalities and estimates
  should be considered true \emph{for $m$ sufficiently
  large}. Since $\frac{m^{2 + \sigma}}{\Delta_m} 2^{-\Delta_m}
  \leq 1$ and $h (\Delta_m) > 1$, we have
  \[
   2 \varepsilon_m r_m =
   2\frac{h(\Delta_m)\Delta_m}{m^{2+\sigma}}\geq 2 \frac{2^{p_1}
     - 1}{2^{p_2} - 1}.
  \]
  We can therefore estimate that the number of intervals in
  (\ref{eq:numberIntervals}) is bounded from below by
  \[
  (1 - 2 (1 + \varepsilon_m) r_m )
  \cdot \frac{2^{p_2}-1}{2^{p_1}-1}.
  \]
  Hence, the number of intervals where the function $G_{p_1,m}
  \cdot G_{p_2, m}$ is zero, but $G_{p_1,m}$ is not, is at least
  \[
  (2^{p_1} - 1) \biggl( (1 - 2 (1 + \varepsilon_m) r_m ) \cdot
  \frac{2^{p_2}-1}{2^{p_1}-1} \biggr) = (1 - 2 (1 +
  \varepsilon_m) r_m ) \cdot (2^{p_2}-1 ).
  \]
  The total length of these intervals is therefore at least
  \[
  \frac{2 r_m}{2^{p_1}-1} (1 - 2 (1 + \varepsilon_m) r_m ) \cdot
  (2^{p_2}-1 ) \geq 2 r_m (1 - 2 (1 + \varepsilon_m) r_m).
  \]
    
  Continuing like this, we observe that the product
  $\prod^m_{k=1} G_{k,m}(x)$ is $0$ on a length of at least
  \begin{align*}
    2r_m \cdot \sum^{m/\Delta_m}_{k=1} (1 - 2 (1 + \varepsilon_m)
    r_m )^{k-1} & = 2r_m \cdot \frac{1 - (1 - 2 (1 +
      \varepsilon_m) r_m )^{\frac{m}{\Delta_m}}}{2 (1 +
      \varepsilon_m) r_m} \\ & = \frac{1}{1+\varepsilon_m} \cdot
    \bigl( 1 - (1 - 2 (1 + \varepsilon_m) r_m
    )^{\frac{m}{\Delta_m}} \bigr) \\ & \geq
    \frac{1}{1+\varepsilon_m} \cdot \bigl( 1 - (1 - 2 r_m
    )^{\frac{m}{\Delta_m}} \bigr) \\ & \geq 1- \varepsilon_m,
  \end{align*}
  provided that 
  \begin{equation}  \label{eq:cond}
    \lim_{m \to \infty} \bigl(1 - 2 r_m
    \bigr)^{\frac{m}{\Delta_m}} = 0.
  \end{equation}

  Under condition (\ref{eq:cond}) we therefore obtain 
  \[
  \mu \bigl( \complement C_m \bigr) \leq \int
  \prod^{m/\Delta_m}_{k=1} G_{k,m}(x) \; \dd x < \varepsilon_m.
  \]
  Since $\sum_m \varepsilon_m < \infty$, we have $\mu (A_n) \to 1$ as
  $n \to \infty$. Hence, $\mu(\ear)=1$.
  
  To conclude we note that condition (\ref{eq:cond}) is satisfied
  for the choice $r_m = \frac{\Delta_m h(\Delta_m)}{m}$, since
  \begin{align*}
    \lim_{m \to \infty} \bigl( 1 - 2 r_m
    \bigr)^{\frac{m}{\Delta_m}} = \lim_{m \to \infty} \biggl(
    \biggl( 1 - 2 \frac{\Delta_m h (\Delta_m)}{m}
    \biggr)^{\frac{m}{\Delta_m h (\Delta_m)}} \biggr)^{h
      (\Delta_m)} = 0
  \end{align*}
  because of $\Delta_m \to \infty$ and 
  \[
  \lim_{m \to \infty} \biggl( 1 - 2 \frac{\Delta_m h
    (\Delta_m)}{m} \biggr)^{\frac{m}{\Delta_m h (\Delta_m)}} =
  \exp (-2) < 1. \qedhere
  \]
\end{proof}

\begin{proof}[Proof of part (\ref{thm:EAROne}) in Theorem~\ref{thm:EAR}]
  The choice $\Delta_m= (2+\sigma)\log_2(m)$ satisfies the
  assumptions of Proposition \ref{prop:Ear} since
  $\frac{m}{(2+\sigma)\log_2(m)} \to \infty$ and
  \[
  \frac{m^{2+\sigma}}{\Delta_m} 2^{-\Delta_m} =
  \frac{1}{(2+\sigma)\log_2(m)}<1. \qedhere
  \] 
\end{proof}

\subsection{Sufficient condition for measure zero}

In the converse direction, part (\ref{thm:EARZero}) in
Theorem~\ref{thm:EAR} follows from the next proposition.

\begin{prop}
  Let $X=[0,1]$, $T(x)=2x \mod 1$ and let $\mu$ denote the
  Lebesgue measure. Suppose that
  \begin{equation} \label{eq:cond2}
    \lim_{m\to \infty}m r_m =0.
  \end{equation}
  Then $\mu(\ear)=0$.
\end{prop}

\begin{proof}
  Identifying $[0,1]$ with the circle, the set $E_{k,m}$ consists
  of $2^k - 1$ intervals (sectors) of length $2 r_m / (2^k -
  1)$. Hence the measure of $E_{k,m}$ is $2 r_m$.

  It follows immediately that $\mu (C_m) \leq 2 m r_m$. Hence,
  the condition
  \[
  \lim_{m \to \infty} m r_m = 0
  \]
  implies that $\mu (\ear) = 0$.
\end{proof}

\begin{remark}
  From the equation $\mu \bigl( \complement C_m \bigr) \geq 1 - 2r_m
  m$ we can also deduce the necessary condition for $\mu (\ear ) =1$ that
  \[
  \lim_{m \to \infty} \mu(B_m) m  =\lim_{m \to \infty} 2r_m m \geq 1
  \]
  In particular, $\ear$ cannot have
  full measure for $\mu(B_m) = \frac{c}{m}$ with any $c<1$.
\end{remark}

\end{document}